\documentclass[a4paper,10pt]{amsart}
\usepackage{amssymb,amsmath,amsfonts,amsthm}
\usepackage[dvips]{graphicx}
\usepackage{psfrag}
\usepackage{color}
\usepackage{todonotes}

\theoremstyle{plain}
\newtheorem{main}{Theorem}

\newtheorem{maincor}[main]{Corollary}
\newtheorem{theorem}{Theorem}[section]
\newtheorem{lemma}[theorem]{Lemma}
\newtheorem{proposition}[theorem]{Proposition}

\theoremstyle{remark}
\newtheorem{remark}[theorem]{Remark}
\newtheorem{definition}[theorem]{Definition}

\newcommand\numberthis{\addtocounter{equation}{1}\tag{\theequation}}

\newcommand{\quand}{\quad\text{and}\quad}
\newcommand{\Leb}{\operatorname{vol}}

\newcommand{\C}{\operatorname{C}}

\newcommand{\diam}{\operatorname{diam}}
\newcommand{\G}{\operatorname{G}}

           \def\ea{\end{array}}
          \def\ec{\end{center}}
     \def\ed{\end{description}}
        \def\ee{\end{equation}}
       \def\eea{\end{eqnarray}}
     \def\eeaa{\end{eqnarray*}}
 \def\et{\end{thebibliography}}
\def\bib{\bibitem}

\def\Orb{{\rm Orb}}
\def\Diff{{\rm Diff}}
\def\Cl{{\rm Cl}}
\def\Gibb{{\rm Gibb}}
\def\inv{{\rm inv}}

\def\supp{\operatorname{supp}}

\def\cG{{\mathcal G}}
\def\cA{{\mathcal A}}

\def\cC{{\mathcal C}}

\def\cI{{\mathcal I}}

\def\cU{{\mathcal U}}
\def\cV{{\mathcal V}}
\def\cR{{\mathcal R}}
\def\cB{{\mathcal B}}
\def\cH{{\mathcal H}}

\def\cF{{\mathcal F}}
\def\cM{{\mathcal M}}

\def\cR{{\mathcal R}}

\def\cS{{\mathcal S}}

\def\loc{\operatorname{loc}}

\def\vep{\varepsilon}

\def\TT{{\mathbb T}}

\title[New criterion]{New criterion of physical measures for partially hyperbolic diffeomorphisms}

\author{Yongxia Hua, Fan Yang and Jiagang Yang}
\date{\today}
\thanks{Y.H. is supported
by NSFC No. 11401133. J.Y. is partially supported by CNPq, FAPERJ, and PRONEX}

\address{Department of Mathematics, Southern University of Science and Technology of China, 1088 Xueyuan Rd., Xili, Nanshan District, Shenzhen, Guangdong, China 518055}
\address{Department of Mathematics, University of Oklahoma, Norman, Oklahoma, USA}
\address{Department of Mathematics, Southern University of Science and Technology of China, Guangdong, China; and 
Departamento de Geometria, Instituto de Matem\'atica e Estat\'istica, Universidade
Federal Fluminense, Niter\'oi, Brazil}
\email{yangjg\@@impa.br}

\begin{document}

\begin{abstract}
We show that for any $C^1$ partially hyperbolic diffeomorphism, there is a full volume subset,
such that any Cesaro limit of any point in this subset satisfies the Pesin formula for partial entropy.

This result has several important applications. First we show that for any $C^{1+}$ partially hyperbolic diffeomorphism with one  dimensional center, there is a full volume subset, such that every point in this set belongs to either the basin of
a physical measure with non-vanishing center exponent, or the center exponent of any  limit of the sequence $\frac1n\sum_{i=0}^{n-1}\delta_{f^i(x)}$ 
is vanishing.

We also prove that for any diffeomorphism with mostly contracting center, it admits a $C^1$ neighborhood
such that every diffeomorphism in a $C^1$ residual subset of this open set admits finitely many physical
measure, whose basins have full volume.
\end{abstract}

\maketitle

\section{Introduction}
\emph{Partially hyperbolic} diffeomorphisms were proposed by Brin, Pesin~\cite{BP} and Pugh,
Shub~\cite{PS72} independently at the early 1970's, as an extension of the class of Anosov
diffeomorphisms \cite{A,AS}. A diffeomorphism $f$ being partially hyperbolic means that there exists a
decomposition $TM = E^s \oplus E^c \oplus E^u$ of the tangent bundle $TM$ into three
continuous invariant sub-bundles: $E^s_x$, $E^c_x$ and $E^u_x$, such that $Tf \mid E^s$ is
uniform contraction, $Tf\mid E^u$ is uniform expansion and $Tf \mid E^c$ lies in between:
$$
\frac{\|Tf(x)v^s\|}{\|Tf(x)v^c\|} \le \frac 12
\quand
\frac{\|Tf(x)v^c\|}{\|Tf(x)v^u\|} \le \frac 12
$$
for any unit vectors $v^s\in E^s$, $v^c\in E^c$, $v^u\in E^u$ and any $x\in M$.

Partially hyperbolic diffeomorphisms form an open subset of the space of $C^r$ diffeomorphisms
of $M$, for any $r\ge 1$. The \emph{stable sub-bundle} $E^s$ and the \emph{unstable sub-bundle}
$E^u$ are uniquely integrable, that is, there are unique foliations $\cF^s$ and $\cF^u$
whose leaves are smooth immersed sub-manifolds of $M$ tangent to $E^s$ and $E^u$, respectively,
at every point.

The \emph{partial entropy} of an invariant probability measure for unstable foliation
is a value to quantify the complexity of the measure generated
on this foliation, whose precise definition will be given in Section~\ref{ss.partialentropy} (see also~\cite{VY3}). Let $\mu$
be any invariant measure of $f$, denote the partial entropy of $\mu$ along the unstable foliation
$\cF^u$ by $h_\mu(f,\cF^u)$.

Let $f$ be a partially hyperbolic diffeomorphism. Denote by
$$\G^u(f)=\{\mu; h_\mu(f,\cF^u)-\int \log \det(Tf\mid_{E^u(x)})d \mu(x)\geq 0\}.$$
Recall that if $f$ is $C^{1+\alpha}$, then by Ruelle's inequality:
$$
h_\mu(f,\cF^u)-\int \log \det(Tf\mid_{E^u(x)})d \mu(x)\leq 0.
$$
A similar result for $C^1$ diffeomorphisms with dominated splitting is proved in~\cite{WWZ}. We will see that for all $C^1$ diffeomorphism $f$, the set $\G^u(f)$ is always non-empty, compact and convex (Proposition~\ref{p.Gu}).

For a diffeomorphism $f:M\to M$ on some compact Riemannian manifold $M$.
An invariant probability $\mu$ is a \emph{physical measure} for $f$ if the
set of  points $z\in M$ for which
\begin{equation}\label{eq.SRBmeasure}
\frac 1n \sum_{i=0}^{n-1} \delta_{f^i(z)} \to \mu
\quad \text{(in the weak$^*$ sense)}
\end{equation}
has positive volume. This set is denoted by $Basin(\mu)$ and called the \emph{basin}
of $\mu$. 

Following Pesin and Sinai~\cite{PS82} and Bonatti and Viana~\cite{BV00} (see also~\cite[Chapter 11]{BDVnonuni}),
we call \emph{Gibbs $u$-state} any invariant probability measure whose conditional probabilities
(see Rokhlin~\cite{Rok49}) along strong unstable leaves are absolutely continuous with respect to the
Lebesgue measure on the leaves. The set of Gibbs $u$-states plays important roles in the study of physical measures for partially hyperbolic diffeomorphisms. In particular, it is shown in~\cite{BDVnonuni} that if $f$ is $C^{1+\alpha}$, then every ergodic  Gibbs $u$-state such that all the center Lyapunov exponents are negative must be a physical measure. 

A program for investigating the physical measures of partially hyperbolic
diffeomorphisms was initiated by Alves, Bonatti, Viana in \cite{ABV00,BV00},
who proved the existence and finiteness when $f$ is $C^{1+\alpha}$ and is either ``mostly expanding''
(asymptotic forward expansion) or ``mostly contracting'' (asymptotic forward
contraction) along the center direction.

While the study on the existence and finiteness of physical measures for $C^{1+\alpha}$ systems is fruitful, the results for $C^1$ systems are surprisingly lacking. This is mainly due to:
\begin{itemize}
\item lack of candidate measures, since Gibbs $u$-states are to the Pesin's formula, which requires higher regularity; 
\item when considered as a equilibrium state, the associated potential function $\phi=-\log \det Tf|_{E^u}$ is not H\"older; in particular, the distortion is unbounded;
\item the Pesin stable lamination is not absolutely continuous for $C^1$ diffeomorphisms.
\end{itemize}
As far as the authors know, \cite{CQ} and~\cite{Q} are  the only existing work on the existence of physical measures for $C^1$ dynamical systems, where the systems in question are uniformly expanding/hyperbolic.

The main theorem in this paper states  that measures in $\G^u(f)$ are the natural candidates for physical measures.
\begin{main}\label{main.newcriterion}

Let $f$ be a $C^1$ partially hyperbolic diffeomorphism. Then there is a full volume subset $\Gamma$ such that
for any $x\in \Gamma$, any limit of the sequence $\frac{1}{n}\sum_{i=0}^{n-1}\delta_{f^i(x)}$ belongs to $\G^u(f)$.
\end{main}
\begin{remark}\label{rk.uniqueG}
As a direct consequence, if $\G^u(f)$ consists of a unique measure, then this measure is automatically a physical measure whose basin has full volume.
\end{remark}

The partially hyperbolic diffeomorphisms with
\emph{mostly contracting center}, which were first studied
in~\cite{BV00}, are those $\C^{1+}$ partially hyperbolic diffeomorphisms whose
Gibbs $u$-states only have negative center Lyapunov exponents. 

Mostly contracting diffeomorphisms contain an abundance of systems. See for example~\cite{BV00} and~\cite{DVY}. It is also shown in~\cite{Y16} that mostly contracting is a $C^1$ open property. Moreover, if $f$ is a $C^{1+}$ partially hyperbolic volume preserving diffeomorphisms with one dimensional center, such that the volume has negative center exponent and the unstable foliation is minimal, then $f$ is mostly contracting. 

Our next theorem generalizes the main result of~\cite{DVY} to generic $C^1$ diffeomorphisms.

\begin{main}\label{main.mostlycontracting}
Let $f$ be a $C^{1+}$ partially hyperbolic diffeomorphism with mostly contracting center.
Then it has a $C^1$ neighborhood $\cU$ and a residual subset of diffeomorphism $\cR\subset \cU$,
such that every $f\in \cU$ has only finitely many physical measures, whose basins cover a set with full volume.
\end{main}
In fact, we will give a more precise description about the number of physical measures and the structure of their support for diffeomorphisms $C^1$ close to a mostly contracting system. See Theorem~\ref{t.mostlyoontracting}.

A similar result for partially hyperbolic diffeomorphisms with mostly expanding center also holds,
as shown in~\cite{Y18}. In particular, if $f$ is partially hyperbolic with one dimensional center, accessible and volume preserving, such that the volume has positive center Lyapunov exponent, then $C^1$ generic diffeomorphisms $g$ in a neighborhood of $f$ has a unique Physical measure, whose basin has full volume. This provides a measure theoretic way to prove robust transitivity, see~\cite[Theorem F]{Y16}. See~\cite[Corollary E, Theorem F]{Y16}. The proof there also makes use of our new criterion in Theorem~\ref{main.newcriterion}.

Denote by $PH_1$ the set of $C^1$  partially hyperbolic diffeomorphisms with one dimensional center, and let $PH_1^r = PH_1\cap \Diff^r(M)$. Our next theorem gives a complete description for the structure of the Cesaro limit of typical points for diffeomorphisms in $PH_1^{1+}$.

Recall that almost every ergodic component of a Gibbs $u$-state is again a Gibbs $u$-state. We denote by $\Gibb^{u,0}$ the invariant measures whose ergodic decomposition
are supported on ergodic Gibbs $u$-states with zero  center exponent.

\begin{main}\label{main.center1d}
Let $f\in PH_1$ be a $C^{1+}$ partially hyperbolic diffeomorphism. Then there is a full volume subset $\Gamma$, such that for any $x\in \Gamma$, we have
\begin{itemize}
\item[(i)] either $x$ is in the basin of a physical measure $\mu^-$, where $\mu^-$ is an ergodic Gibbs $u$-state with negative center exponent;

\item[(ii)] or $x$ is in the basin of a physical measure $\mu^+$, where $\mu^+$ is an ergodic Gibbs $u$-state with positive center exponent and is a Gibbs $cu$-state, i.e., the conditional measures of $\mu^+$ along center unstable leaves are absolutely continuous with respect to the volume there;

\item[(iii)] or every Cesaro limit of $\frac{1}{n}\sum_{i=0}^{n-1}\delta_{f^i(x)}$ belongs to $\Gibb^{u,0}$.
\end{itemize}
\end{main}
\begin{remark}\label{rk.uniquevanishing}
The previous theorem shows that, at least for $C^{1+}$ partially hyperbolic diffeomorphism,
the main difficulty to prove the existence of a physical measure is the existence of
Gibbs $u$-states with vanishing center exponent. As an immediate corollary, we get:
\end{remark}
\begin{maincor}\label{maincor.uniquevanishing}
Let $f$ be a $C^{1+}$ partially hyperbolic diffeomorphism with one dimensional center. If $f$ has at most one ergodic Gibbs $u$-state with vanishing center exponent, then it admits physical measures, whose basins have full volume.
\end{maincor}
As a simple example, let us take an linear Anosov diffeomorphism $A$ on the torus $\TT^2$.  Given  $p\in\TT^2$ that is a  fixed point of $A$, we take a family of $C^2$ diffeomorphisms  $\{h_x:S^1\to S^1\}_{x\in\TT^2}$, such that
$$
h_p(\theta) =\theta+\alpha(\theta),
$$
with $\alpha(\theta)$ satisfying  (when considered as a function on [0,1])
$$
\alpha(0)=\alpha(1)=0,\alpha(\theta)>0 \text{ for }\alpha\neq 0 \text{ or }1.
$$
In other words, $h_p$ rotate every point counterclockwise except $0$, where $0$ is the unique fixed point of $h_p$. We also assume that $h_x(\theta)$ is $C^2$ in  $x$.

Then we consider the $C^2$ diffeomorphism $f:\TT^2\times S^1\to\TT^2\times S^1$ given by the skew product:
$$
f(x,\theta)=(A(x), h_x(\theta)).
$$ 
We can take $h_x$ such that $f$ is partially hyperbolic with one dimensional center. By the invariance principle (see~\cite{VY1}), every ergodic Gibbs $u$-state $\mu$ of $f$ with vanishing center exponent must be $su$-invariant, which means that $\mu$ has a family of continuous conditional measures along center leaves, which is invariant under the $su$-holonomy map. Moreover, the conditional measure of $\mu$ along the center leaf $\{p\}\times S^1$ consists of exactly one point, namely $(p,0)$. By the $su$-invariance, on every center leaf, the conditional measure of $\mu$ is supported at exactly one point. This shows that the support of $\mu$ is an invariant torus of $f$. It then follows that $\Gibb^{u,0}(f)$ consists of a unique measure, which can be seen as the lift of the Lebesgue measure  on $\TT^2$ to the invariant torus. By Corollary~\ref{maincor.uniquevanishing}, $f$ has physical measures, whose basin covers a set with full volume.
 
We finish this section with the following Corollary:

\begin{maincor}\label{maincor.finitephysicalmeasure}
Let $f\in PH_1$ be a $C^{1+}$ partially hyperbolic diffeomorphism. If every ergodic Gibbs $u$-state of $f$ has non-vanishing center exponent, then $f$ has finitely many physical measures, whose basin covers a full volume set.
\end{maincor}
 
This paper is organized in the following way: Section~\ref{s.preliminary} contains the background on foliation charts and the partial entropy along expanding foliations. In Section~\ref{s.structure}, we show that for every $C^1$ partially hyperbolic diffeomorphism $f$, the set $\G^u(f)$ is non-empty, compact and convex. Furthermore, it varies upper semi-continuously  with $f$ in the $C^1$ topology. The main tool there is the upper semi-continuity of the partial entropy along $\cF^u$, proved in~\cite{Y16}.

Section~\ref{s.proof} consists of the proof of Theorem~\ref{main.newcriterion}. In particular, we show that the measure of points whose Cesaro average is away from $\G^u(f)$ has exponentially small measure. The proof uses an argument similar to the proof of the variational principle.  In Section~\ref{s.uniform} we make the observation that the argument in~\ref{s.proof} can be made uniform. 

In Section~\ref{s.mostlycontracting}, we prove Theorem~\ref{main.mostlycontracting} using a topological structure of  partially hyperbolic diffeomorphism known as the skeleton. This is previously used in~\cite{DVY} to study the Gibbs $u$-states for $C^{1+}$ mostly contracting diffeomorphisms, and is generalized here to the $C^1$ case. Finally, we prove Theorem~\ref{main.center1d}, Corollary~\ref{maincor.uniquevanishing} and~\ref{maincor.finitephysicalmeasure} in Section~\ref{s.1d} by introducing a set $\G^{cu}(f)$ which is defined similar to $\G^u(f)$ using the Pesin's formula. 

While preparing this paper, we are kindly notified by C. Crovisier, D. Yang and J. Zhang that a similar project~\cite{CYZ} is being developed.  In particular,~\cite[Theorem A and C]{CYZ} are similar to the Theorem~\ref{main.newcriterion} and~\ref{main.center1d} here. We would also  like to point out that the techniques used in~\cite{CYZ} are different from those in this paper.

\section{Preliminary}\label{s.preliminary}
In this section, we introduce some necessary backgrounds on the partial entropy, which is the main tool in this paper.

\subsection{Foliation charts\label{ss.charts}}
Let $f$  be a partially hyperbolic diffeomorphism over a $d$-dimensional Riemannian manifold $M$, and $\cF^u$ be its unstable foliation with dimension $l$.
In this subsection we explain the foliation charts for the unstable foliation.

A \emph{$\cF^u$-foliation
box} is the image $B$ of a topological embedding $\Phi : D^{d-l}\times D^{l}\to M$
such that every plaque $P_x = \Phi(\{x\}\times D^l)$ is contained in a leaf of $\cF^u$,
and every
$$\Phi(x,\cdot) : D^l \to M, y \mapsto \Phi(x,y)$$
is a $\C^1$ embedding that depends continuously on $x$ in the $\C^1$ topology.
We write $D=\Phi(D^{d-l}\times \{0\})$, and denote this foliation
box by ($B, \Phi, D$). We cover the manifold $M$ by a finite covering of foliation boxes
$\{(B_i,\Phi_i,D_i)\}_{i=0}^{k_0}$ and assume that the Lebesgue number of this covering is larger than 1.

For any partition $\cA$ with diameter smaller than the Lebesgue number of the covering by foliation boxes, the intersection of elements in $\cA$ with local unstable plaques  $P_x$ induces a uncountable partition, which we denote by  $\cA^u$.

For a given $n>0$,
we are going to consider the following partitions:
$$\cA_{n}=\bigvee_{i=0}^{n} f^{-i}(\cA).$$
They also induce partitions $(\cA_{n})^u$ which is the intersection of elements of $\cA_n$ with local unstable plaques $P_x$. 

The proof the next lemma is straight forward:

\begin{lemma}\label{l.localskew}
$\cA_{n}^u=\bigvee_{i=0}^{n}\cA^u$.
\end{lemma}
The foliation boxes induce two sequence of finite partitions $\cA_{n,m}$ and $\cA_{0,m}$ such that:
\begin{itemize}
\item $\cA_{n}=\cA_{n,0}\prec\cA_{n,1}\prec \cdots$ and $\cA=\cA_{0,0}\prec\cA_{0,1}\prec\cdots$

\item $\cA_{n}\prec \cA_{n,m}\prec \cA_{n}^u$ and $\cA\prec \cA_{0,m}\prec \cA^u$;

\item $\bigvee_{m} \cA_{n,m}=\cA_{n}^u$ and $\bigvee_{m} \cA_{0,m}=\cA^u$.
\end{itemize}
More precisely, define a sequence of partition $\overline{\cC}_m$ on $\cup_i D_i$ such that
$$\overline{\cC}_1\prec \cdots \prec \overline{\cC}_m\prec \cdots$$
with diameter of $\overline{\cC}_m$ converges to zero. For every element $\overline{C}$ of the partition $\overline{\cC}_m$, we may consider
$C=\cup_{x\in \overline{C}}\cF^u_{loc}(x)$.
For every element $P$ of $\cA$, fix an $i=i(P)$ such that $P$ is contains in the foliation box $(B_i,\Phi_i,D_i)$. Then
the $C\cap P$ induces an element of the partition $\cA_{0,m}$. Moreover, for every element $Q\in \cA_{n}$, the same construction yields a partition $\cA_{n,m}$, which refines both $\cA_{0,m}$ and $\cA_n$. Indeed, we have:
\begin{equation}\label{eq.oneside}
\cA_{n} \bigvee \cA_{0,m}=\cA_{n,m}.
\end{equation}

%In each foliation box $(B_i,\Phi_i,D_i)$, there is a projection map $\pi_i$ from $B_i$ to the transverse disk $D_i$ along the
%plaques of foliation. In particular,
%$$\pi_i(\cA_{0,m})=C_m.$$
%For $\rho>0$, we denote by $\cA(0,m, \rho)$ the \emph{transverse $\rho$ neighborhood} of $\cA$ which are the
%point of $\cA_{0,m}$ under projection $\pi_i$ stays distance at most $\rho$ with the boundary of $C_{m}$.
\subsection{Measurable partitions and mean conditional entropy}
Let $\cB$ be the Borel $\sigma$-algebra on $M$. In this subsection, we recall
the properties of a measurable partition. For more details, see~\cite{Rok49, Rok67}.

\begin{definition}\label{d.measurablepartition}

A partition $\xi$ of $M$ is called \emph{measurable} if there is a sequence of finite partitions
$\xi_n$ $_{n\in \mathbb{N}}$ such that:
\begin{itemize}
\item elements of $\xi_n$ are measurable; %(up to $\mu$-measure $0$);

\item $\xi=\vee_{n}\xi_n$, that is, $\xi$ is the coarsest partition which refines $\xi_n$ for each $n$.

\end{itemize}
\end{definition}

For a partition $\xi$ and $x\in M$, we denote by $\xi(x)$ the element of $\xi$ which contains $x$.
Given any measurable partition $\xi$ and a probability measure $\mu$, we may define the conditional measures of $\mu$ on almost every element.

\begin{proposition}\label{p.conditionalmeasure}
Let $\xi$ be a measurable partition. Then there is a full $\mu$-measure subset $\Gamma$
such that for every $x\in \Gamma$, there is a probability measure $\mu_x^{\xi}$ defined
on $\xi(x)$, satisfying:
\begin{itemize}
\item Let $\cB_\xi$ be the sub-$\sigma$-algebra of $\cB$ which consists of unions of
      elements of $\xi$, then for any measurable set $A$, the function $x\to \mu^\xi_x(A)$ is
      $\cB_\xi$-measurable.

\item Moreover, we have
\begin{equation}\label{eq.conditional}
\mu(A)=\int \mu^\xi_x(A)d\mu(x).
\end{equation}
for every $\cB$  measurable  set $A$.
\end{itemize}
\end{proposition}
\begin{remark}
Let $\pi_\xi$ be the projection $M\to M/\xi$, and $\mu_\xi$ be the projection of measure $\mu$ onto
$M/\xi$ by the map $\pi_\xi$. Then equation~\eqref{eq.conditional} can be written as:
\begin{equation}\label{eq.projective}
\mu(A)=\int \mu^\xi_B(A)d\mu_\xi(B)
\end{equation}
where $B$ denotes the element of $\xi$ and $\mu^\xi_B$ the conditional measure on $B$.
\end{remark}

Let $\xi$ be a measurable partition and $C_1,C_2,\dots$ be the elements of $\xi$ with positive $\mu$ measure.
We define the \emph{entropy of the partition} by
\begin{equation}
H_\mu(\xi) = \begin{cases} \sum_k \phi(\mu(C_k)), & \mbox{if } \mu(M\setminus \cup_k C_k)=0 \\
\infty, & \mbox{if } \mu(M\setminus \cup_k C_k)>0 \end{cases}
\end{equation}
where $\phi: \mathbb{R}^+\to \mathbb{R}$ is defined by $\phi(x)=-x\log x$.

If $\xi$ and $\eta$ are two measurable partitions, then for every element $B$ of $\eta$, $\xi$
induces a partition $\xi_B$ on $B$. We define the \emph{mean conditional
entropy of $\xi$ respect to $\eta$}, denoted by $H(\xi\mid\eta)$, as the following:
\begin{equation}\label{eq.conditionalentropy}
H_\mu(\xi\mid\eta)=\int_{M/\eta} H_{\mu^\eta_B}(\xi_B) d\mu_\eta(B).
\end{equation}
\begin{definition}
For measurable partitions $\{\zeta_{n}\}_{n=1}^\infty$ and $\zeta$, we write $\zeta_n\nearrow \zeta$ if
the following conditions are satisfied:
\begin{itemize}
\item $\zeta_1<\zeta_2<\dots$;
\item $\vee_{n=1}^\infty \zeta_n=\zeta$.
\end{itemize}
\end{definition}

\begin{lemma}\label{l.increasingpartition}\cite[Subsection 5.11]{Rok67}
Suppose $\{\eta_n\}_{n=1}^\infty$, $\eta$ and $\xi$ are measurable partitions, such that $\eta_n\nearrow \eta$
and $H_\mu(\xi\mid \eta_1)<\infty$, then
$$H_\mu(\xi\mid \eta_n)\searrow H_\mu(\xi \mid \eta).$$
\end{lemma}

\begin{definition}
Let $\xi$ be a measurable partition, we put
$$h_\mu(f,\xi)=H_\mu(\xi\mid f\xi^+),$$
where $\xi^+=\vee_{n=1}^\infty f^n \xi$.
\end{definition}

\begin{remark}\label{r.increasingpartition}
A measurable partition $\xi$ is said to be \emph{increasing} if $f\xi <\xi$. For
an increasing partition $\xi$,
$$h_\mu(f,\xi)=H_\mu(\xi\mid f\xi).$$
\end{remark}

\subsection{Partial entropy\label{ss.partialentropy}}

Throughout this subsection, $\cF^u$ denotes the unstable foliation of $f$. We are going to give
the precise definition of the partial metric entropy of $\mu$ along
$\cF^u$, which is defined using  a special class of measurable partitions:

\begin{definition}
We say that a measurable partition $\xi$ of $M$ is \emph{$\mu$-subordinate} to the unstable foliation if
for $\mu$-a.e. $x$, we have
\begin{itemize}
\item[(1)] $\xi(x)\subset \cF^u(x)$ and $\xi(x)$ has uniformly small diameter inside $\cF^u(x)$;
\item[(2)] $\xi(x)$ contains an open neighborhood of $x$ inside the leaf $\cF^u(x)$;
\item[(3)] $\xi$ is an increasing partition, meaning that $f\xi \prec \xi$.
\end{itemize}

\end{definition}

Ledrappier, Strelcyn~\cite{LS82} proved that the Pesin unstable lamination admits some $\mu$-subordinate
measurable partition. More precisely, they take a finite partition $\cA$ (indeed induced by a covering consisting of
finitely many balls) with 'small boundary' in the sense that
\begin{equation}\label{eq.smallboundary}
\sum_{i=0}^\infty \mu(B_{\lambda^i}(\partial\cA))<\infty
\end{equation}
for some $\lambda\lesssim 1$, then $\bigvee_{i=0}^\infty f^i(\cA^u)$ is a $\mu-$subordinate partition. (See for instance Yang~\cite{Y16}).

The following result %(for the subordinate partitions constructed as in Section~\ref{s.construction})
is contained in Lemma~3.1.2 of Ledrappier, Young~\cite{LY85a}:

\begin{lemma}\label{l.definitionleafentropy}
$h_\mu(f,\xi_1)=h_\mu(f,\xi_2)$
for any measurable partitions $\xi_1$ and $\xi_2$ that are $\mu$-subordinate to $\cF^u$.
\end{lemma}

This allows us to give the following definition:

\begin{definition}\label{d.partialentropy}

The \emph{partial $\mu$-entropy} $h_\mu(f,\cF^u)$ of the unstable foliation $\cF^u$
is defined by $h_\mu(f,\xi)$ for any $\mu$-subordinate partition $\xi$.
\end{definition}

Let $\cA$ be a finite partition and define $\cA^u_n$, $\cA_{n,m}$ as in Section~\ref{ss.charts}, we assume further that it satisfies \eqref{eq.smallboundary}. The following proposition plays an important role in the study of partial entropy:
\begin{proposition}\label{p.approximate}\cite[Proposition 4.1, 4.5]{Y16}
\item[(i)] $\frac{1}{n}H_\mu(\cA_{n}^u\mid \cA^u)\searrow h_\mu(f,\cF^u)$;
\item[(ii)] for any $n>0$, $H_\mu(\cA_{n,m}\mid \cA_{0,m})\searrow H_\mu(\cA_{n}^u\mid \cA^u)$ as $m\to\infty$.
\end{proposition}

\section{The structure of $\G^u(f)$}\label{s.structure}
Recall that the space $\G^u$ is defined as:
$$
\G^u(f)=\{\mu\in \cM_{\inv}(f); h_\mu(f,\cF^u)-\int \log \det(Tf\mid_{E^u(x)})d \mu(x)\geq 0\}.
$$
In this section we will prove the following properties of $\G^u$:
\begin{proposition}\label{p.Gu}
Let $f$ be a $C^1$ partially hyperbolic diffeomorphism. Then $\G^u(f)$ is non-empty, weak* compact and convex. Typical ergodic components
of measure of $\G^u(f)$ are also in $\G^u(f)$.
\end{proposition}

Before proving this proposition, let us state the following properties for the partial entropy along $\cF^u$.

\begin{theorem}\cite[Theorem A]{Y16}\label{t.semicontinuous}
Let $f_n$ be a sequence of $\C^1$ diffeomorphisms which converge to $f$ in the $\C^1$ topology,
and $\mu_n$ invariant measures of $f_n$ which converge to an invariant measure $\mu$ of $f$ in the
weak* topology. Suppose $\cF_n$ is an expanding foliation of $f_n$ for each $n$, with $\cF_n\to \cF$, then
$$\limsup h_{\mu_n}(f_n,\cF^u_n)\leq h_\mu(f,\cF^u_f).$$
\end{theorem}

\begin{theorem}\cite[Theorem A]{WWZ}\label{t.Ruelle}
Let $f$ be a $C^1$ partially hyperbolic diffeomorphisms. Then for any invariant measure $\mu$, we have
$$h_\mu(f,\cF^u)\leq\int \log \det(Tf\mid_{E^u(x)})d \mu(x).$$
\end{theorem}

Recall that for a $C^{1+}$ diffeomorphism $f$, a measure $\mu$ is called a Gibbs $u$-state if the conditional measures of $\mu$ along unstable leaves are absolutely continuous with respect to the volume there. Denote by $\Gibb^u(f)$ the set of Gibbs $u$-states of $f$. The following
basic properties of Gibbs $u$-states can be found in Bonatti, D\'{i}az and Viana~\cite[Subsection 11.2]{BDVnonuni}:
\begin{proposition}\label{p.Gibbsustates} Let $f$ be a $C^{1+}$ partially diffeomorphism. Then:
\begin{itemize}
\item[(1)] $\Gibb^u(f)$ is non-empty, weak* compact and convex. Ergodic components of Gibbs $u$-states are Gibbs u-states.
\item[(2)] The support of every Gibbs $u$-state is $\cF^u$-saturated, that is, it consists of entire strong
    unstable leaves.
\item[(3)] For Lebesgue almost every point $x$ in any disk inside some strong unstable leaf, every accumulation point of
    $\frac{1}{n}\sum_{j=0}^{n-1}\delta_{f^j(x)}$ is a Gibbs $u$-state.
\item[(4)] Every physical measure of $f$ is a Gibbs $u$-state and, conversely, every ergodic $u$-state whose center Lyapunov
    exponents are negative is a physical measure.
\end{itemize}
\end{proposition}

Now we are ready to prove Proposition~\ref{p.Gu}.

\begin{proof} The main tool of the proof is the upper semi-continuity of partial entropy given by Theorem~\ref{t.semicontinuous}.\\
\noindent (1).  $\G^u(f)$ is non-empty for every $f$ that is $C^1$.

We take a sequence of $C^{1+}$ partially hyperbolic diffeomorphisms $f_n\to f$ in $C^1$ topology. By Ledrappier~\cite{L84}, the space $\G^u(f_n)$ coincides with the space of Gibbs $u$-states $\Gibb^u(f_n)$, which is non-empty by Proposition~\ref{p.Gibbsustates} (1). Take $\mu_n\in\G^u(f_n)$. Passing to a subsequence if necessary, we may assume that $\mu_n\to \mu\in \cM_{\inv}(f)$ in the weak* topology.

By Theorem~\ref{t.semicontinuous}, we have
\begin{align*}
h_\mu(f,\cF^u)\geq&\limsup h_{\mu_n}(f_n,\cF^u_n)\\
=&\limsup \int \log \det(Tf_n\mid_{E_n^u(x)})d \mu_n(x)\\
=&  \int \log \det(Tf\mid_{E^u(x)})d \mu(x),
\end{align*}
which shows that $\mu\in G ^u(f)$.

\noindent (2). Compactness.

Assume there is a sequence of invariant measures $\mu_n\in \G^u(f)$ and they converge to $\mu$ in weak* topology. From the definition of $\G^u$, we have for each $n$:
\begin{align*}
h_{\mu_n}(f,\cF^u)&\geq \int \log(\det(Tf\mid_{E^u(x)}))d\mu_n(x).
\end{align*}

By Theorem~\ref{t.semicontinuous}, we have
$$\limsup_{n\to \infty}h_{\mu_n}(f,\cF^u)\leq h_{\mu}(f,\cF^u).$$
As a consequence, we have
\begin{align*}
h_{\mu}(f,\cF^u)&\geq \int \log(\det(Tf\mid_{E^u(x)}))d\mu(x).
\end{align*}
This implies $\mu\in \G^u(f)$.

\noindent (3). Convexity. 

Convexity follows immediately from the fact that the partial entropy is an affine functional. See~\cite[Proposition 2.14]{HHW}.

\noindent (4). Ergodic component.  

Assume that there is a set $\Gamma$ with positive $\mu$ measure, such that for every $x\in \Gamma$, 
$\mu_x\notin \G^u(f)$ where $\mu_x$ is the ergodic component of $\mu$ at $x$. Then for every $x\in \Gamma$ we must have  
$$
P_{\mu_x} :=h_{\mu_x}(f,\cF^u)-\int \log \det(Tf\mid_{E^u(y)})d \mu_x(y)< 0.
$$
By Theorem~\ref{t.semicontinuous} and  \cite[Proposition 2.14]{HHW}, $P_{\mu_x}$ is affine and upper semi-continuous. As a result, 
$$
0 \leq P_\mu: = h_{\mu}(f,\cF^u)-\int \log \det(Tf\mid_{E^u(y)})d \mu(y) = \int P_{\mu_x} d\mu(x).
$$
As a result, there must be some ergodic component $\mu_x$ of $\mu$ such that $P_{\mu_x}>0$. This contradicts Theorem~\ref{t.Ruelle}.

The proof of Proposition~\ref{p.Gu} is now complete. 
\end{proof}

Next, we show that $\G^u(f)$ is upper semi-continuous with respect to $f$ in $C^1$ topology. To this end, we consider the map $\cG^u$ from the space of partially hyperbolic  diffeomorphisms  to the compact subspace of probabilities of $M$ given by:
$$\cG^u(g)\to \G^u(g).$$

\begin{proposition}\label{p.Guppercontinuous}
The map $\cG^u$ is upper semi-continuous with respect to the weak* and the $C^1$ topology.
\end{proposition}
\begin{proof}
The proof is similar to the proof of the non-emptiness of $\G^u$. We need to show that
for any sequence of $C^1$ partially hyperbolic diffeomorphisms $g_n$ with  $g_n\overset{C^1}{\longrightarrow} g$, we have
$$\limsup_{n\to \infty} \G^u(g_n)\subset \G^u(g).$$
It is equivalent to show that for any  $\mu_n\in \G^u(g_n)$ converging to $\mu$ in the weak-* topology,
we have $\mu\in \G^u(g)$.

From the definition of space $\G^u(\cdot)$, we have for each $n$:
\begin{align*}
h_{\mu_n}(g_n,\cF^u_{g_n})&\geq \int \log(\det(Tg_n\mid_{E_{g_n}^u(x)}))d\mu_n(x).
\end{align*}

By Theorem~\ref{t.semicontinuous}, we have
$$\limsup_{n\to \infty}h_{\mu_n}(g,\cF^u_g)\leq h_{\mu}(g,\cF^u_g).$$

On the other hand, for partially hyperbolic diffeomorphisms, the invariant bundles vary continuously with respect to the diffeomorphisms. This gives
\begin{align*}
\lim_{n\to \infty} \int \log(\det(Tg_n\mid_{E^u_{g_n}(x)}))d\mu_n(x) &= \int \log(\det(Tg\mid_{E^u(x)}))d\mu(x).
\end{align*}
As a result, we have:
\begin{align*}
h_{\mu}(g,\cF^u_g)&\geq \int \log(\det(Tg\mid_{E^u(x)}))d\mu(x).
\end{align*}
This implies $\mu\in \G^u(g)$. The proof is complete.

\end{proof}

\section{Proof of Theorem~\ref{main.newcriterion}}\label{s.proof}
Throughout this section, let $I$ be a disk transverse to the bundle $E^{cs}$, and $\phi^u(x)=-\log \det(Tf\mid_{E^u(x)})$. We will also write 
$$
\phi^u_n = \sum_{i=0}^{n-1}\phi^u\circ f^i
$$
for the ergodic sum of $\phi^u$.

Let us briefly explain the structure of the proof. Note that the space of probabilities of $M$ and space of the invariant measure  of $f$ are both metric spaces with the same metric. Let $d$ be any metric on these spaces that generates the weak* topology. We will take a $\vep$-neighborhood of $\G^u$ in the space of invariant measures. The complement of this neighborhood is compact, and can be finitely covered by a collection of open and convex balls $U_i$, which are taken inside the space of probability measures (not necessarily invariant).

Then for every $U_i$, we consider  the set of  $x$ such that $\frac{1}{n}\sum_{j=0}^{n-1} \delta_{f^j(x)}$ (note that this measure is not invariant; this is why we need to take $U_i$ in the space of all probability measures) is contained in $U_i$. The main result (Theorem~\ref{t.exponentialtail}) in this section is that the set described above has exponentially small measure. By Borel-Contelli, for almost every $x$, any Cesaro limit cannot be contained in $U_i$. Since there are only finitely many such balls, this shows that the Cesaro limit of typical points must be close to $\G^u(f)$ (Theorem~\ref{t.convergence}). 

To this end, for any fixed $\vep>0$, consider the compact subset of invariant measures:
$$\cB_{\vep}=\{\mu\in\cM_{\inv}(f): d(\mu,\G^u(f))\geq \vep\}.$$

By upper semi-continuity of partial entropy (\cite[Theorem A]{Y16}), there is $\alpha>0$ such that for every $\mu\in \cB_{\vep}$,
\begin{equation}\label{eq.pressuregap}
h_\mu(f,\cF^u)+\int \phi^u d\mu < -\alpha.
\end{equation}

For every $\mu\in \cB_\vep$, take a finite partition $\cA=\cA_\mu$ with 'small boundary' and small diameter. Here
small boundary means that $\cA$ satisfies \eqref{eq.smallboundary}, which implies in particular that
\begin{equation}\label{eq.vanishingboundary}
\mu(\partial(\cA))=0;
\end{equation}
and small diameter means
$$\diam(\cA)<\delta$$
where $\delta>0$ is taken small enough such that
\begin{equation}\label{eq.smalldiameter}
|\log \det(Tf\mid_{E^u(x)})-\log \det(Tf\mid_{E^u(y)})|\leq \alpha/8
\end{equation}
for any two points $x,y$ such that $y\in \cA^u(x)$.

For every $\mu\in \cB_\vep$, by Proposition~\ref{p.approximate}[(i)],
$$\frac{1}{n}H_\mu(\bigvee_{i=1}^{n}f^{-i}(\cA^u)\mid \cA^u)\searrow h_\mu(f,\cF^u),$$
we can take  $n_{\mu}\in \mathbb{N}$ such that
$$\frac{1}{n_\mu}H_\mu(\bigvee_{i=1}^{n_\mu}f^{-i}(\cA^u)\mid \cA^u)+\int \phi^u d\mu < -\alpha,$$
and
\begin{equation}\label{eq.log2}
\frac{\log 2}{n_\mu}<\alpha/8.
\end{equation}

Define a sequence of partitions $\{\overline{\cC}_m\}$ on the transverse disks of the foliations charts as in Section~\ref{ss.charts}.
We may assume that
\begin{equation}\label{eq.transverseboundary}
\mu(\cup_{x\in \partial(\overline{\cC}_m)} \cF^u_{loc}(x))=0.
\end{equation}
Recall that for every element $\overline{C}_m$ of $\overline{\cC}_m$, we denote by $C_m=\cup_{x\in \overline{C}_m}\cF^u_{loc}(x)$.

As in Section~\ref{ss.charts}, the intersection of $C_m$ with elements $A\in \cA$ induces a partition $\cA_{0,m}$.
Similarly, the intersection of  $C_m$ with elements of $\cA_{n_\nu} = \bigvee_{i=0}^{n_{\nu}}f^{-i}(\cA)$,
induces a partition we denote by $\cA_{n_{\nu},m}$.

by Proposition~\ref{p.approximate}[(ii)],
$$H_\mu(\cA_{n,m}\mid \cA_{0,m})\searrow H_\mu(\bigvee_{i=1}^{n}f^{-i}(\cA^u)\mid \cA^u).$$
We may fix $m_\mu$
such that 
$$\frac{1}{n_\mu}H_\mu( \cA_{n_\mu,m_\mu}\mid \cA_{0,m_\mu})+\int \phi^u d\mu<-\alpha.$$

Note that $\cA_{n_\mu,m_\mu}$ and $\cA_{0,m_\mu}$ are both finite measurable partitions, there is a convex neighborhood $U_\mu$ of $\mu$ in the space of probability measures (not necessarily invariant), such that
every $\nu\in U_\mu$ satisfies the following three conditions,
\begin{itemize}
\item[(a)] \begin{equation}\label{eq.finit}
\frac{1}{n_\mu}H_\nu(\cA_{n_\mu,m_\mu}\mid \cA_{0,m_\mu})+\int \phi^u d\nu < -\alpha.
\end{equation}
\item[(b)]
\begin{equation}\label{eq.smallboundaryentropy}
\frac{1}{n_\mu}\left(-\sum_{A\in \cA_{n_\mu,m_\mu}}a_A\log b_A-\sum_{B\in \cA_{0,m_\mu}}c_B \log d_B\right)+\int \phi^u d\nu<-\alpha
\end{equation}
where $a_A = \nu(\Cl(A))$, $b_A=\nu(int(A))$, and  $c_B=\nu(\Cl(B))$, $d_B=\nu(int(B))$.
\item[(c)] There is a neighborhood $V_{\mu}$ of $\partial(\cA_{n_\mu,m_\mu})$ such that
\begin{equation}\label{eq.boundarytail}
\nu(V_\mu)\log\#(\cA_{n_\mu,m_\mu})<\frac{\alpha}{8}.
\end{equation}
\end{itemize}

Note that measures $\nu\in U_\mu$ may assign positive weight on the boundary of $\cA_{n_{\nu},m_\mu}$.
When this happens, conditions (b) and (c) are used to control the entropy of $\nu$ near the boundary of $\cA_{n_{\nu},m_\mu}$. We may further assume that $U_{\mu}$ are balls around $\mu$ in the set of probability measures.

As a direct consequence of \eqref{eq.smallboundaryentropy}, we get:
\begin{lemma}\label{l.limitinneighborhood}
For any $\nu\in U_\mu$, and any sequence of measures $\mu_j\to \nu$,
$$\limsup_j \frac{1}{n_\mu}H_{\mu_j}(\cA_{n_\mu,m_\mu}\mid \cA_{0,m_\mu}) + \int \phi^u d\mu_j< -\alpha.$$
\end{lemma}
Fix a $\vep>0$. The balls $U_{\mu_i}$ described above is also open in the space of invariant measures, and forms an open covering of $\cB_{\vep}$. Since $B_\vep$ is compact, we may take a finite sub-covering $\{U_i = U_{\mu_i}\}_{i=1}^N$. Each $U_i$ is associate with $\cA_i=\cA_{\mu_i}$, $n_i=n_{\mu_i}, m_i=m_{\mu_i}$, and $V_i=V_{\mu_i}$ the boundary of $\cA_{n_i,m_i}$ that satisfies \eqref{eq.boundarytail}.

Recall that $I$ is a disk transverse to the bundle $E^c$. Fix any ball $U_i$ which is associated with $\mu_i\in \cM_{\inv}$. For every $n>0$, denote by
$$
B_{n,i}=\{x \in I: \frac{1}{n}\sum_{j=0}^{n-1} \delta_{f^j(x)}\in U_i\}.
$$
The next theorem states that points in $I$ whose Cesaro average is in $U_i$ has exponentially small measure. For this purpose, denote by $\Leb$ the Lebesgue measure on $I$.

\begin{theorem}\label{t.exponentialtail}
$\limsup \frac{1}{n}\log \Leb(B_{n,i})<-\alpha/4$.
\end{theorem}
In particular, we have $\sum_{n=0}^\infty\Leb(B_{n,i})<\infty$. By the Borel-Contelli lemma, $\Leb$-almost every $x\in I$ is contained in $B_{n,i}$ for only finitely many $n$. This shows that there exists $\Gamma_i\subset I$ with full Lebesgue measure, such that for every $x\in \Gamma_i$, any limit of  $\frac{1}{n}\sum_{j=0}^{n-1} \delta_{f^j(x)}$ is not contained in $U_i$.  Take $\Gamma=\bigcap_{i=1}^N\Gamma_i$ which still has full Lebesgue measure, we obtain the following theorem:

\begin{theorem}\label{t.convergence}
For every $\vep>0$ and every disk $I$ transverse to $E^{cs}$, there is a full Lebesgue measure subset $\Gamma\subset I$, such that for every $x\in\Gamma$ and any Cesaro limit $\mu$ of the sequence $ \frac{1}{n}\sum_{j=0}^{n-1} \delta_{f^j(x)}$, we have 
$$
d(\mu,\G^u(f)<\vep).
$$ 
\end{theorem}

\begin{proof}[Proof of Theorem~\ref{t.exponentialtail}]
For simplicity, we will only consider the case $i=1$ from now on. Consider $\cA=\cA_1$ the corresponding finite partition, denote
by $\cA^u$ the local intersection of $\cA$ with unstable leaves as describe in Section~\ref{ss.charts}. 
%Fix $\delta$ given by \eqref{eq.smalldiameter}.

The proof is motivated by the proof of the variational principle. For every $\epsilon>0$ and each $n>0$,
we consider the $(n,\delta)$-separated set $E_n$ of $B_{n,1}$, such that
$$\sum_{x\in E_n}e^{\phi^u_n(x)}\geq\sup_{E\text{ is a }(n, \delta)\text{-separated set }}\sum_{x\in E}e^{\phi^u_n(x)}-\epsilon.$$

We need the following lemma, whose proof is in the next subsection.
\begin{lemma}\label{l.exponentialtail}
$\limsup\frac{1}{n}\log \sum_{x\in E_n}e^{\phi^u_n(x)}<-\alpha/2$.
\end{lemma}
Let us continue the proof. Denote by $B_n(x,\delta)$ the $(n,\delta)$-Bowen ball of $x$. By
definition, $\cup_{x\in E_n} B_n(x,\delta)$ covers $B_{n,1}$.

Recall that  $I$ is a disk that is  transverse to the $E^{cs}$ bundle. By the dominated splitting between $E^{cs}$ and $E^u$, after finite steps of iteration by $f$,
its tangent bundle will be almost tangent to the bundle $E^u$. Replacing $I$ by its iteration,
we may assume that the tangent bundle of $I$ is almost tangent to $E^u$ bundle. By \eqref{eq.smalldiameter},
for any $x\in I$ and $y\in B_n(x,\delta)$, we have
\begin{equation}
|\log \det(Tf^n\mid_{T_xI})-\log \det(Tf^n\mid_{T_yI})|\leq n\alpha/4.
\end{equation}

Then we have  $$\Leb(B_n(x,\delta))<e^{n\alpha/4}\det(Tf^{-n}\mid_{TI_x})\Leb(f^n(B_n(x,\delta))).$$
Since $f^n(B_n(x,\delta))$ is contained in the ball with center $f^n(x)$ and
radius $\delta$, the Lebesgue measure of  $f^n(B_n(x,\delta))$ must be uniformly bounded:  there is $K>0$  independent of $x$ and $n$, such that
$$\Leb(f^n(B_n(x,\delta)))\leq K.$$
Thus,
\begin{equation}
\Leb(B_n(x,\delta))<Ke^{n\alpha/4}\det(Tf^{-n}\mid_{TI_x})=Ke^{n\alpha/4}e^{\phi^u_n(x)}.
\end{equation}
Sum over $x\in E_n$, we obtain
\begin{equation*}
\Leb(B_{n,1})\leq \sum_{x\in E_n}\Leb(B_n(x,\delta))\leq Ke^{n\alpha/4}\sum_{x\in E_n} e^{\phi^u_n(x)}.
\end{equation*}
By Lemma~\ref{l.exponentialtail}, we obtain:
$$\limsup \frac{1}{n}\log \Leb(B_{n,1})\leq \alpha/4-\frac{1}{2}\alpha\leq -\alpha/4.$$
The proof is complete.
\end{proof}
\subsection{Proof of Lemma~\ref{l.exponentialtail}}
\begin{proof}
Denote by
$$\nu_{n,0}=\frac{1}{\sum_{x\in E_n}e^{\phi^u_n(x)}}\sum_{x\in E_n} e^{\phi^u_n(x)} \delta_x,$$
and $f^k_*(\nu_{n,0})=\nu_{n,k}$.
Then 
$$\mu_n=\frac{1}{n}\sum_{i=0}^{n-1}\nu_{n,i}$$
is a convex combination of measures of the form $\frac{1}{n}\sum_{j=0}^{n-1} \delta_{f^j(x)}$ with  $x\in E_n \subset B_{n,1}$. Since $U_1$ is convex, we have 
$$\mu_n \in U_1.$$

Because $diam(\cA)\leq \delta$, every element of $\bigvee_{i=0}^{n-1} f^{-i}\cA$ contains at most one element
of $E_n$, we have
\begin{equation}\label{eq.pressure}
\log \sum_{x\in E_n}e^{\phi^u_n(x)}=H_{\nu_{n,0}}(\bigvee_{i=0}^{n-1}f^{-i}\cA )+\int \phi^u_n d\nu_{n,0}=H_{\nu_{n,0}}(\bigvee_{i=0}^{n-1}f^{-i}\cA )+n\int \phi^u d\mu_n.
\end{equation}
Denote partition $\cI_k=\{f^k(I),(f^k(I))^\circ\}$. Because $\nu_{n,k}=f^k_*(\nu_{n,0})$ is supported on $f^k(I)$, $\cI_k$ is
the trivial partition for $\nu_{n,k}$ in the sense that every element has measure $0$ or $1$. In particular,
$$H_{\nu_{n,0}}(\bigvee_{i=0}^{n-1}f^{-i}\cA )=H_{\nu_{n,0}}(\bigvee_{i=0}^{n-1}f^{-i}\cA\mid\cI_0).$$

Similar to the proof of the variational principle, for every $j=0,\cdots, n_1-1$, write $a(j)=[(n-j)/n_1]$. Fix $0\leq j \leq n_1-1$. We have
$$\bigvee_{i=0}^{n-1}f^{-i}\cA=\bigvee_{i=0}^{j-1}f^{-i}\cA \vee \bigvee_{r=0}^{a(j)-1}f^{-(rn_1+j)}(\bigvee_{i=0}^{n_1-1} f^{-i}\cA) \vee \bigvee_{i=a(j)n_1}^{n-1}f^{-i}\cA$$
and the `head' and the `tail' given by $S=\{0,\cdots, j-1\}\cup\{a(j)n_1,\cdots, n-1\}$ has cardinality at most $2n_1$.
Therefore,
\begin{align*}
H_{\nu_{n,0}}(\bigvee_{i=0}^{n-1}f^{-i}\cA\mid \cI_0)&= H_{\nu_{n,0}}(\bigvee_{i=0}^{j-1}f^{-i}\cA\mid \cI_0)\\
&+ \sum_{r=0}^{a(j)-1}H_{f^{rn_1+j}_*\nu_{n,0}}(\bigvee_{i=0}^{n_1-1}f^{-i}\cA\mid \cI_{rn_1+j} \bigvee_{i=1}^{rn_1+j}f^i\cA )\\
&+H_{f^{a(j)n_1+j}_*(\nu_{n,0})}(\bigvee_{i=a(j)n_1}^{n-1}f^{-i}\cA\mid \cI_{a(j)n_1+j}\vee \bigvee_{i=1}^{a(j)n_1+j}f^{-i}\cA).\\
&=I+II+III,
\end{align*}
where $I+III\leq H_{\nu_{n,0}}(\bigvee_{i=1}^{j-1}f^{-i}\cA)+H_{\nu_{n,f^{a(j)n_1+j}}}(\bigvee_{i=a(j)n_1}^{n-1}f^{-i}\cA)\leq 2n_1\log \#(\cA)$.

Summing over $j=0,\cdots, n_1-1$, we obtain
\begin{align*}
n_1H_{\nu_{n,0}}(\bigvee_{j=0}^{n-1}\cA\mid \cI_0)&\leq 2n_1^2 \log\#A\\\numberthis \label{eq.19}
 &\;\;\;\;+\sum_{k=0}^{n-1}H_{\nu_{n,k}}(\bigvee_{i=0}^{n_1-1}f^{-i}\cA\mid \cI_k\vee \bigvee_{i=1}^{k}f^i\cA ).\\
\end{align*}

From now on, we estimate
\begin{align*}
H_{\nu_{n,k}}(\bigvee_{i=0}^{n_1-1}f^{-i}\cA\mid \cI_k\vee \bigvee_{i=1}^{k}f^i\cA )&=H_{f^{-1}_*\nu_{n,k}}(\bigvee_{i=1}^{n_1}f^{-i}\cA\mid \cI_{k-1}\vee \bigvee_{i=0}^{k-1}f^i\cA ).
\end{align*}
Because for partitions $\cC\prec \cA$, $H_\mu(\cB\mid \cA)=H_\mu (\cB\bigvee \cC\mid \cA)$, we have
\begin{align*}
H_{\nu_{n,k}}(\bigvee_{i=0}^{n_1-1}f^{-i}\cA\mid \cI_k\vee \bigvee_{i=1}^{k}f^i\cA )&=H_{\nu_{n,k-1}}(\bigvee_{i=1}^{n_1}f^{-i}\cA \vee \cA \mid \cI_{k-1}\vee \bigvee_{i=0}^{k-1}f^i\cA )\\
&=H_{\nu_{n,k-1}}(\bigvee_{i=0}^{n_1}f^{-i}\cA \mid \cI_{k-1}\vee \bigvee_{i=0}^{k-1}f^i\cA ).
\end{align*}

Note that the diameter of $f^k(I)\in\cI_k$ becomes unbounded for $k$ large. Denote by $\tilde{\cI}_k$  the partition induced by the intersection of elements of $\cI_k$ with the foliation box (thus cutting $f^k(I)$ into `local leaves')  and with
elements of $\cA$. Recall that we take the partition $\cA$ to have diameter less than the Lebesgue number of the covering by foliation boxes, thus every element of $\cA$ is contained in a foliation box. This gives the following  lemma, whose proof can be found in~\cite[Lemma 4.2]{Y16}.
\begin{lemma}
$\cI_k\vee \bigvee_{i=0}^{k}f^i\cA \succ \tilde{\cI}_k$.
\end{lemma}

\subsubsection{An easy case: $I\subset \cF^u$}
When the disk $I$ is taken to be part of a local $\cF^u$ leaf, the proof is easy. By the previous lemma, we have:

\begin{align*}
&H_{\nu_{n,k}}(\bigvee_{i=0}^{n_1}f^{-i}\cA\mid \cI_k \vee \bigvee_{i=0}^{k-1}f^i\cA )\leq H_{\nu_{n,k}}(\bigvee_{i=0}^{n_1}f^{-i}\cA\mid \tilde{\cI}_k)\\
&\leq H_{\nu_{n,k}}(\bigvee_{i=0}^{n_1}f^{-i}\cA\mid \cA_{0,m_1})\\
&\le H_{\nu_{n,k}}(\cA_{n_1,m_1}\mid \cA_{0,m_1}).
\end{align*}
Thus~\eqref{eq.19} becomes:
\begin{align*}
&n_1H_{\nu_{n,0}}(\bigvee_{j=0}^{n-1}\cA\mid \cI) \\\numberthis\label{eq.20}
&\leq2n_1^2\log \#(\cA)+\sum_{k=0}^{n-1} H_{\nu_{n,k}}(\cA_{n_1,m_1}\mid \cA_{0,m_1}).
\end{align*}
Divide both sides by $nn_1$ and combine this with~\eqref{eq.pressure}. Note that $\frac{n_1}{n}\log\#\cA\to0$ as $n\to\infty$, we have
\begin{align*}
&\limsup\frac1n\log \sum_{x\in E_n}e^{\phi^u_n(x)}\\=&\limsup\left(\frac1nH_{\nu_{n,0}}(\bigvee_{i=0}^{n-1}f^{-i}\cA )+\int \phi^u d\mu_n\right)\\
\leq& \limsup_n \Bigg(\frac{1}{n_1}H_{\mu_n}(\cA_{n_1,m_1}\mid \cA_{0,m_1}) + \frac{2n_1}{n}\log\#\cA+\int \phi^u d\mu_n\Bigg)\\
<& -\alpha,
\end{align*}
where the last step is given by Lemma~\ref{l.limitinneighborhood}. This finishes the proof of Lemma~\ref{l.exponentialtail} when $I$ is contained in a local $\cF^u$ leaf.
\subsubsection{General case: $I$ is not contained in a local $\cF^u$ leaf}
In this case, elements of $\cI_k$ will have non-trivial intersection with the boundary of $\cA_{n_1,m_1}$, which are local $\cF^u$ leaves. However, if $k$ is large enough, elements of $\cI_k$ will be `almost tangent' with $\cF^u$ leaves. In particular, when a non-trivial intersection occurs, the corresponding element of $\cI_k$ must be contained in a small neighborhood of the boundary of $\cA_{n_1,m_1}$, which has small entropy.

For this purpose, we denote  by $\partial^T\cA_{0,m_1},\partial^T\cA_{n_1,m_1}$ the transverse boundary of the partitions $\cA_{0,m_1}$ and $\cA_{n_1,m_1}$, i.e.,
the intersections of elements in $\cA$ and $\cA_{n_1}$  with $\cup_{x\in \partial \overline{C}}\cF^u_{loc}(x)$, where
$\overline{C}$ is an element of $\cC_m$. For any $\rho>0$, we also consider the
$\rho$-transverse neighborhood of $\cA_{0,m_1}$ and $\cA_{n_1,m_1}$, which we denote by
$\partial_{\rho}^T\cA_{0,m_1}$ and $\partial_{\rho}^T\cA_{n_1,m_1}$, as the intersection of elements $\cA$ and $\cA_{n_1}$
with $\cup_{x\in \partial_\rho \overline{C}}\cF^u_{loc}(x)$, where $\partial_{\rho} \overline{C}$ is the
$\rho$ neighborhood of the boundary of $\overline{C}$. In other words, $\partial_{\rho}^TA_{0,m_1}\in\partial_{\rho}^T\cA_{0,m_1}$ can be seen as the $\rho$-neighborhood of $\partial^TA_{0,m_1}\in\partial^T\cA_{0,m_1}$, and the same can be said about $\partial_{\rho}^T\cA_{n_1,m_1}$.

When $\rho$ is sufficiently small,
\begin{equation}\label{eq.smalltransverse}
\partial^TA_{0,m_1}\subset \partial A_{0,m_1},\partial^TA_{n_1,m_1}\subset \partial A_{n_1,m_1} \text{ and } \partial^T_{\rho}A_{0,m_1},\partial^T_{\rho}A_{n_1,m_1}\subset V_1,
\end{equation}
where $V_1$ is the neighborhood of $\partial(\cA_{n_1,m_1})$ given by~\eqref{eq.boundarytail}.

Because $f^k(I_k)$ approaches uniformly to the unstable foliation, there is $k_1>0$ such that for any $k>k_1$,
any element $I_k$ of the partition $\cI_k$ and any element $A_{0,m_1}$  of partition $\cA_{0,m_1}$, if $I_k$ has non-empty intersection
with $A_{0,m_1}$ outside $\partial^T_{\rho}A_{0,m_1}$, then
$I_k\subset A_{0,m_1}$.

Thus the partition $\tilde{\cI}_k$ induces two sub-partitions: the `good' part $\tilde{\cI}_{k,g}$ consists of the elements of $\tilde{\cI}_k$
which is not complete contained in an element of $\partial^T_{\rho}\cA_{0,m_1}$ (i.e., those that do not intersects with the boundary), and the `bad' part$\tilde{\cI}_{k,b}$ the elements of $\tilde{\cI}_k$ that is contained in elements of  $\partial^T_{\rho}\cA_{0,m_1}$. By the  discussion above, we have
\begin{lemma}\label{l.touchbound}
For $k>k_1$, $\tilde{\cI}_{k,g}\succ \cA_{0,m_1}$ outside $\partial^T_{\rho}\cA_{0,m_1}$.
\end{lemma}

Since the measure $\nu_{n,k}$ is supported on $f^k(I)$, we can take $\tilde{I}_{k,g}$ the union of elements of $\tilde{\cI}_{k,g}$ and write $\nu_{n,k}=g_{n,k}\nu_{n,k,g}+(1-g_{n,k})\nu_{n,k,b}$,
where $g_{n,k}=\nu_{n,k}(\tilde{I}_{k,g})$ close to $1$ is the portion of the `good' part of $\cI_k$, and $\nu_{n,k,g} = \nu_{n,k}|_{\tilde{I}_{k,g}}$. 
Then from the definition, we have

\begin{lemma}
For $k>k_1$, and any element $A_{n_1,m_1}$  of $\cA_{n_1,m_1}$,
$$g_{n,k}\nu_{n,k,g}(A_{n_1,m_1}\setminus \partial^T_{\rho}A_{n_1,m_1})=\nu_{n,k}(A_{n_1,m_1}\setminus \partial^T_{\rho}A_{n_1,m_1}).$$
\end{lemma}

Then for $k>k_1$:
\begin{align*}
&H_{\nu_{n,k}}(\bigvee_{i=0}^{n_1}f^{-i}\cA\mid \cI_k \vee \bigvee_{i=0}^{k-1}f^i\cA )\leq H_{\nu_{n,k}}(\bigvee_{i=0}^{n_1}f^{-i}\cA\mid \tilde{\cI}_k)\\
&\leq g_{n,k} H_{\nu_{n,k,g}}(\bigvee_{i=0}^{n_1}f^{-i}\cA\mid \tilde{\cI}_{k,g})
+(1-g_{n,k})H_{\nu_{n,k,g}}(\bigvee_{i=0}^{n_1}f^{-i}\cA\mid \tilde{\cI}_{k,b})+ \log 2\\
&\leq g_{n,k} H_{\nu_{n,k,g}}(\bigvee_{i=0}^{n_1}f^{-i}\cA\mid \cA_{0,m_1})
+(1-g_{n,k})n_1\log \#(\cA)+ \log 2\\
&= g_{n,k} H_{\nu_{n,k,g}}(\cA_{n_1,m_1}\mid \cA_{0,m_1})+(1-g_{n,k}) n_1\log \#(\cA)+ \log 2.
\end{align*}

Thus
\begin{align*}
&n_1H_{\nu_{n,0}}(\bigvee_{j=0}^{n-1}\cA\mid \cI) \leq\\
&2n_1^2\log \#(\cA)+ n\log 2+\sum_{k=0}^{n-n_1-1}(1-g_k)n_1\log \#(\cA)\\
&+\sum_{k\leq k_1} H_{\nu_{k,g}}(\cA_{n_1,m_1}\mid \cA_{0,m_1})+\sum_{k=0}^{n-n_1-1}g_{n,k} H_{\nu_{n,k,g}}(\cA_{n_1,m_1}\mid \cA_{0,m_1})\\
&\leq 2n_1^2\log \#(\cA)+ n\log 2+k_1 \log \#\cA_{n_1,m_1}\\
& +\sum_{k=0}^{n-n_1-1}(1-g_k)n_1\log \#(\cA)+\sum_{k=0}^{n-n_1-1}g_{n,k} H_{\nu_{n,k,g}}(\cA_{n_1,m_1}\mid \cA_{0,m_1}).
\end{align*}
In view of~\eqref{eq.19}, the three extra terms here corresponds to the boundary of $\cA_{n_1,m_1}$.
Observe that
$\frac{1}{nn_1}(2n_1^2\log \#(\cA)+k_1 \log \#\cA_{n_1,m_1})\to 0$ and $\frac{\log 2}{n_1}<\alpha/8$.
And
$$\frac{1}{nn_1}(\sum_{k=0}^{n-n_1-1}(1-g_k)n_1\log \#(\cA))\leq \frac{\log \#(\cA)}{n}(\sum_{k=0}^{n-1}(1-g_k))$$
Because $\mu_n=\frac{1}{n}\sum_{k=0}^{n-1}\nu_{n,k}\in U_1$, and $1-g_k\leq\nu_{n,k}(V_1) $, thus by \eqref{eq.boundarytail},
$$\frac{1}{nn_1}(\sum_{k=0}^{n-n_1-1}(1-g_k)n_1\log \#(\cA))\leq \mu_n(V_1)\log \#(\cA)<\alpha/4.$$

It remains to estimate $$\frac{1}{nn_1}\sum_{k=0}^{n-n_1-1}g_k H_{\nu_{k,g}}(\cA_{n_1,m_1}\mid \cA_{0,m_1}).$$

Because $\mu_n=\frac{1}{n}\sum_{k=0}^{n-1}\nu_{n,k}=\frac{1}{n}\sum_{k=0}^{n-1}g_{n,k}\nu_{n,k,g}+\frac{1}{n}\sum_{k=0}^{n-1}(1-g_{n,k})\nu_{n,k,b}$,
$w_n=\frac{1}{n}\sum_{k=0}^{n-1}\nu_{n,k,b}$ is a probability measures. Denote by $a_n=(1-\frac{\sum g_k}{n})$, then $a_n\leq \mu_n(V_1)$.
\begin{align*}
&\frac{1}{nn_1}\sum_{k=0}^{n-1}g_{n,k} H_{\nu_{n,k,g}}(\cA_{n_1,m_1}\mid \cA_{0,m_1})\\
&\leq \frac{1}{n_1}H_{\mu_n}(\cA_{n_1,m_1}\mid \cA_{0,m_1})-\frac{1}{n_1} a_n H_{w_n}(\cA_{n_1,m_1}\mid \cA_{0,m_1})\\
&\leq \frac{1}{n_1}H_{\mu_n}(\cA_{n_1,m_1}\mid \cA_{0,m_1}).
\end{align*}

By Lemma~\ref{l.limitinneighborhood}, taking to the limit, we have
\begin{align*}
&\limsup_n\frac{1}{n}\log \sum_{x\in E_n}e^{\phi^u_n(x)}\\
=&\limsup_n \left(\frac{1}{n}H_{\nu_{n,0}}(\bigvee_{j=0}^{n-1}\cA\mid I^u) + \int \phi^u d\mu_n \right)\\
\leq& \limsup_n \left(\frac{1}{n_1}H_{\mu_n}(\cA_{n_1,m_1}\mid \cA_{0,m_1}) + \int \phi^u d\mu_n\right) +\alpha/8+\alpha/4\\
<&-\alpha/2.\\
\end{align*}
The proof is complete.

\end{proof}

\subsection{Proof of Theorem~\ref{main.newcriterion}\label{ss.newcriterion}}
\begin{proof}
Let $\cF$ be a smooth foliation such that each of its leaf is uniformly transverse to the bundle $E^{cs}$.
Let $\{C_i\}_{i=1}^k$ be a sequence of foliation box which covers the ambient manifold. For any $x\in C_i$, denote
by $I_x$ the plaque of $\cF\mid_{C_i}$ that contains $x$. For any fix $\vep=\frac{1}{m}>0$ and $I_x$, by Theorem~\ref{t.convergence},
there is a full Lebesgue measure subset $\Gamma_{x,m}$ such that for any $y\in \Gamma_{x,m}$, any limit $\mu$ of the sequence
$\frac{1}{n}\sum_{i=0}^{n-1}\delta_{f^i(x)}$ satisfies $d(\mu,\G^u(f))\leq \frac{1}{m}$.
Then there is a full volume subset $\Gamma_x=\cap_m\Gamma_{x,m}\subset I_x$ such that for any $y\in \Gamma_x$, any limit of the sequence
$\frac{1}{n}\sum_{i=0}^{n-1}\delta_{f^i(x)}$ belongs to $\G^u(f)$.

Moreover, since the foliation $\cF$ is smooth, by Fubini theorem, $\cup_x \Gamma_x$ has full volume. 

\end{proof}

\section{Uniform convergence}\label{s.uniform}
In this section, we observe that the convergence in Theorem~\ref{t.exponentialtail} is indeed uniform.

As in Section~\ref{ss.newcriterion}, let $\cF$ be a smooth foliation such that each of its leaves is uniformly transverse to the bundle $E^{cs}$.
Let $\{C_i\}_{i=1}^k$ be a collection of foliation boxes which covers the ambient manifold. For any $x\in C_i$, denote
by $I_x$ the plaque of $\cF\mid_{C_i}$. With the same hypothesis of
Theorem~\ref{t.exponentialtail}, for any $x\in C_i$ and $n>0$, denote by
$B_{x,n,i}=\{y\in I_x; \frac{1}{n}\sum_{j=0}^{n-1} \delta_{f^j(y)}\in U_i\}$.

\begin{proposition}\label{p.uuniform}
$\limsup_{n\to \infty} \sup_x\{\frac{1}{n}\log \Leb(B_{x,n,i})\}<-\alpha/4$.
\end{proposition}

Since the proof is similar to the proof of Theorem~\ref{t.exponentialtail}, we only sketch the proof and highlight the differences.
\begin{proof}
Since there are only finitely many foliation boxes, we may only consider the points of $x$ belonging to a foliation box $C_i$. For simplicity, we suppose $i=1$ in the proof, and consider the ball of probability
neighborhood $U_1$ of the measure $\mu_1$.

We prove by  contradiction. Let $x_n\in C_1$, write 
$$B_{n}=B_{x_n,n,1}=\{y\in I_{x_n}; \frac{1}{n}\sum_{j=0}^{n-1} \delta_{f^j(y)}\in U_1\}.$$

Consider the $\delta$-separated set $E_n\subset B_{n}$ such that
$$\sum_{y\in E_n}e^{\phi^u_n(x)}=\sup_{E\text{ is a }(n, \delta) \text{-separated set of} B_n}\sum_{x\in E}e^{\phi^u_n(x)}$$

Denote by
$$\nu_{n,0}=\frac{1}{\sum_{y\in E_n}e^{\phi^u_n(y)}}\sum_{y\in E_n} e^{\phi^u_n(y)} \delta_y,$$
the measure supported on $I_{x_n}$ and let $f^k_*(\nu_{n,0})=\nu_{n,k}$.
Since $U_1$ is convex,
$$\mu_n=\frac{1}{n}\sum_{i=0}^{n-1}\nu_{n,i} \in U_1.$$

Then all the remaining discussion still works.
\end{proof}

Now we  show that the convergence of Cesaro limit to the space of $\cM_{\inv}$ is uniform.

Recall that 
$$J=[a,b]=\left\{\int \varphi d\mu\right\}_{\mu\in \G^u(f)}.$$
Write $\varphi_n$ for the ergodic sum of $\varphi$. 
The continuous function $\varphi$ and $\vep>0$ induces a neighborhood $U_0$ of $\G^u(f)$ among probabilities,
such that for any $\nu \in U_0$, $d(\int \varphi d\nu, J)< \vep$.

Take $\vep^\prime=\min\{d(\nu,\G^u(f)); \nu\in (U_0)^c\}$. Consider the compact set $\cB_{\vep^\prime}$ as before. we cover
$\cB_{\vep^\prime}$ by $\{U_i\}_{i=1}^k$ where $U_i$ is taken to be a convex ball in the space of probability measure, and obtain an open covering $\{U_1,\cdots, U_k; U_0\}$ of the space
$\cM_{inv}$. By the compactness of $\cM_{\inv}$ in the space of probability measures, we have 
$d(\cM_{\inv}, \left(\cup_{i=0}^k U_i\right)^c)>0$. As a result, there is $n_0$ such that for any $n\geq n_0$ and any point $x\in M$,
$$\frac{1}{n}\sum_{i=0}^{n-1}\delta_{f^i(x)}\in \cup_{i=0}^k U_i.$$
This means, in particular, that for any $x\in M$ and for $n\geq n_0$, if $d(\frac{1}{n}\varphi_n(x), I)\geq \vep$, then
$\frac{1}{n}\sum_{i=0}^{n-1}\delta_{f^i(x)}\in \cup_{i=1}^k U_i$.

The discussion above  also implies a general version of the variational principle for partial entropy, which is shown by Hu, Hua and Wu in~\cite{HHW}.

As before, let $I$ be a disk transverse to the bundle $E^{cs}$, and $\phi(x)\in C^0(M)$ be a potential function which is not necessarily the geometric potential in the previous section.

For a fixed $\delta>0$ small enough and $B_n\subset I$ a sequence of measurable sets,
we consider the $(n,\delta)$-separated set $E_n$ such that
$$\sum_{x\in E_n}e^{\phi^u_n(x)}=\sup_{E\text{ is a }(n, \delta) \text{-separated set of $B_n$}}\sum_{x\in E}e^{\phi^u_n(x)}.$$

Denote by
$$\nu_{n,0}=\frac{1}{\sum_{x\in E_n}e^{\phi^u_n(x)}}\sum_{x\in E_n} e^{\phi^u_n(x)} \delta_x,\;\;  \nu_{n,k}=f^k_*(\nu_{n,0}) \text{ and } \mu_n=\frac1n\sum_{i=0}^{n-1} \nu_{n,i}.$$
Suppose $\mu_n\to \mu$, then main property of this section is the following key step on the proof of variation principle.

\begin{proposition}\label{p.strongvariation}
$\limsup\frac{1}{n}\log \sum_{x\in E_n}e^{\phi^u_n(x)}\leq h_{\mu}(f,\cF^u)+\int \phi^u d\mu$.
\end{proposition}

A similar result was proved in~\cite{HHW} when the disk $I$ is contained in an unstable leaf. Here we make the observation that if $I$ is transverse to the $E^{cs}$ bundle, then after a certain number of iteration by $f$, it will be almost tangent to $E^u$ in the sense that $f^n(I)$ must be contained in a small neighborhood of a $\cF^u$ leaf. Then Proposition~\ref{p.strongvariation} follows using the same argument as in Section~\ref{s.proof}.

\section{Diffeomorphisms with mostly contracting center}\label{s.mostlycontracting}
The main result in this section is Theorem~\ref{t.mostlyoontracting}, which gives a complete description of the number and structure of physical measures for generic $C^1$ diffeomorphisms near a $C^{1+}$ mostly contracting system, which easily implies Theorem~\ref{main.mostlycontracting}.  Let $f: M \to M$ be a $C^{1+}$ partially hyperbolic diffeomorphism with  mostly contracting center, i.e., all the Gibbs $u$-state of $f$ have negative center Lyapunov exponent.  It was shown in \cite{BV00} that:

\begin{proposition}\label{p.mostlycontracting}

$f$ admits finitely many physical measures, such that the union of the basins has full volume.
Moreover, there is a bijective map between physical measures and the ergodic Gibbs $u$-states of
$f$.

\end{proposition}

There are more precise description of the number of physical measures.
We say that a hyperbolic saddle point
has maximum index if the dimension of its stable manifold coincides with the dimension of the center stable
bundle $E^{cs}$. A \emph{skeleton} of $f$ is a collection $\cS = \{p_1, \dots, p_k\}$ of
hyperbolic saddle points with maximum index, satisfying:
\begin{itemize}
\item[(i)] For any $x\in M$ there is $p_i \in \cS$ such that the stable manifold $W^s(\Orb(p_i))$
intersects  transversally with  the unstable leaf $\cF^u(x)$ at some point;
\item[(ii)] $W^s(\Orb(p_i)) \cap W^u(\Orb(p_j )) = \emptyset$ for every $i \neq j$, that is,
the points in $\cS$ have no heteroclinic intersections.
\end{itemize}

Not all partially hyperbolic diffeomorphisms $f$ have a skeleton. However, if $f$ has a skeleton, then all skeletons of $f$ must have the same cardinality.  The relation between a skeleton and physical measures for $f$ was established
in~\cite[Theorem A]{DVY}:

\begin{proposition}\label{p.skeleton}
Let $f$ be a $\C^{1+}$ diffeomorphism with mostly contracting center.
Then $f$ admits a skeleton. Moreover, for any skeleton $\cS = \{p_1,\dots, p_k\}$
of $f$, the number of physical measures is precisely $k=\# \cS$.
\end{proposition}
Indeed, it is shown that the support of each physical measure of $f$ coincides with the closure of $W^u(\Orb(p_i))$ for some $i$.

Let us call \emph{pre-skeleton} any finite collection $\{p_1,\dots, p_k\}$
of saddles with maximum index satisfying condition (i), that is, such that every
unstable leaf $\cF^u(x)$ has a point of transverse intersection with $W^s(\Orb(p_i))$ for
some $i$. Thus a pre-skeleton $S$ is a skeleton if and only if there are no heteroclinic intersections within $S$. More precisely,

\begin{lemma}\label{l.preskeletoncontainskeleton}\cite[Lemma 2.5]{DVY}

Every pre-skeleton contains a subset which is a skeleton.

\end{lemma}

This notion of pre-skeleton is useful, since the continuation of a skeleton or a pre-skeleton under $C^1$ topology is
always a pre-skeleton:

\begin{lemma}\label{l.preskeleton}\cite[Lemma 2.4]{DVY}
Suppose $f$ has a pre-skeleton $\cS = \{p_1, \dots, p_k\}$. Let $p_i(g), i = 1, \dots, k$
be the continuation of the saddles $p_i$ for nearby diffeomorphism $g$. Then
$\cS(g) = \{p_1(g), \dots , p_k(g)\}$ is a pre-skeleton for every $C^1$ diffeomorphism $g$ in a
$C^1$ neighborhood of $f$.
\end{lemma}

\subsection{Upper semi-continuity for the number of physical measures\label{ss.uppersemicontinuous}}
The following result was proved as \cite[Theorem B]{Y16}:

\begin{proposition}\label{p.uppersemicontinuous}
Let $f$ be a $\C^{1+}$ partially hyperbolic diffeomorphism with mostly contracting center.
Suppose $f$ has $k$ physical measures, then there is a $\C^1$ neighborhood $\cU$
of $f$, such that for every $\C^{1+}$ diffeomorphism $g\in \cU$, $g$ has mostly contracting
center, and the number of physical measures of $g$ is at most  $k$.
\end{proposition}

Let us briefly explain the proof of this proposition in~\cite{Y16}. Recall that if $g$ is $C^{1+}$, then the space of Gibbs $u$-state of $g$, denoted by $\Gibb^u(g)$, coincides with $\G^u(g)$. Furthermore, each ergodic Gibbs $u$-state is a physical measure supported on $\overline{\cF^u(\Orb(p))}$ for some hyperbolic  periodic point $p$. Moreover, the collection of  such  periodic points forms a skeleton.  Then it is shown that the cardinality of the skeleton for $C^{1+}$ diffeomorphisms  varies upper semi-continuously with respect to the $C^1$ topology. In particular, the number of ergodic measures in $\Gibb^u(g)$ is bounded by the number of ergodic measures in $\Gibb^u(f)$, for $g$ in a $C^1$ neighborhood of $f$.

Before stating the main theorem of this section, we need the following definition:
\begin{definition}\label{df.Lyapunov}
A set $\Lambda$ of a homeomorphism $f$ is \emph{Lyapunov stable} if there is a sequence of
open neighborhoods $U_1\supset U_2\supset \cdots$ such that:
\begin{itemize}
\item[(a)] $\cap U_i=\Lambda$;
\item[(b)] $f^n(U_{i+1})\subset U_i$ for any $n\geq 0$.
\end{itemize}
\end{definition}

Given a $C^1$ partially hyperbolic diffeomorphism $f$ with mostly contracting center, Denote by $S(f)=\{p_1(f),\ldots,p_{k_f}(f)\}$ a skeleton of $f$. Now we are ready to describe the geometric structure of the physical measures of $f$ $C^1$ close to a mostly contracting system:

\begin{theorem}\label{t.mostlyoontracting}
Let $f$ be a $\C^{1+}$ partially hyperbolic diffeomorphism with mostly contracting center. Then there is a $C^1$ neighborhood $\cU$ of $f$ and a residual set $\cR\in \cU$, such that every $C^1$ diffeomorphism $g\in\cR$ has a skeleton $S(g)=\{p_1(g),\ldots,p_{k_g}(g)\}$ with $k_g\leq k_f$. Moreover, g has exactly $\#S(g)$ physical measures, each of which is supported on $\overline{\cF^u(\Orb(p_i(g)))}$ for some $i$. Furthermore, the basin of each physical measure has full volume within a small neighborhood of its support. Finally, the union of their basin covers a full volume subset of $M$.
\end{theorem}

Clearly, Theorem~\ref{main.mostlycontracting} is a direct consequence of the theorem above.
\subsection{Proof of Theorem~\ref{t.mostlyoontracting}}
Let $f$ be a $C^{1+}$ partially hyperbolic diffeomorphisms with mostly contracting center and let $\cU$ be the $C^1$ neighborhood of $f$ given  by Proposition~\ref{p.uppersemicontinuous}.

Suppose $f$ has $k_f=\#S(f)$ number of  physical measures. For every $1\leq k \leq k_f$, we define the set:
$$\cU^{1+}_k=\{g\in \cU; g \text{ is } C^{1+} \text{ and the number of  physical measures of $g$ is at most $k$} \}.$$

By Proposition~\ref{p.uppersemicontinuous}, the set $\cU^{1+}_k$ is $C^1$ open. We may take $C^1$ open sets $\cU_k$ such that
$\cU^{1+}_k=\cU_k\cap \Diff^{1+}$. Let $\cV_k=(\cU_k\setminus \overline{\cU_{k-1}})$ be the level set of $\cU_k$, then any $C^{1+}$ diffeomorphism
$g\in \cV_k$ has exactly $k$ physical measures. First, we prove the following proposition:

\begin{proposition}
For each $1\leq k\leq k_f$ and  every $C^1$ diffeomorphism $g\in \cV_k$, $g$ has a skeleton $S(g)$ with exactly $k$ elements.
\end{proposition}
\begin{proof}
Fix a $C^{1+}$ diffeomorphism $g_0\in \cV_k$. By the definition of $\cV_k$ and Proposition~\ref{p.skeleton}, $g_0$ has a skeleton with $k$ elements, denote by   $\{p_1(g_0),\cdots,p_k(g_0)\}$.
Shrink $\cU$ if necessary, we may assume that for any $C^1$ diffeomorphism  $g\in \cV_k$, the analytic continuation of periodic points
$\{p_1(g),\cdots, p_k(g)\}$ forms a pre-skeleton of $g$.

We claim $\{p_1(g),\cdots, p_k(g)\}$ is indeed a skeleton of $g$. Otherwise, by Lemma~\ref{l.preskeletoncontainskeleton}, 
the skeleton of $g$ has number of elements strictly less  than $k$, so do diffeomorphisms $C^1$ close to it. Since $\cV_k$ is open, we can take a $C^{1+}$ diffeomorphism $h\in \cV_k$ sufficiently close to $g$. Then the skeleton of $h$ has strictly less than $k$ elements, thus $h$ has strictly less than $k$ physical measures, thanks to Proposition~\ref{p.skeleton}. This contradicts with the definition of $\cV_k$.

\end{proof}

This shows that for $g\in \cV_k$ and $1\leq i \neq j \leq k$, there is $\delta>0$ such that in the $\delta$ neighborhood ($C^1$ topology) of $g$, we have $\cF^u(\Orb(p_i))\cap W^s(\Orb(p_j))=\emptyset$. Using the  connecting lemma, we obtain that for $1\leq i \neq j \leq k$,
\begin{equation}\label{eq.disjoint}
\overline{\cF^u(\Orb(p_i))}\cap \overline{\cF^u(\Orb(p_j))}=\emptyset.
\end{equation}

We need the following generic property proved by Morales and Pacifico~\cite{MP}:

\begin{proposition}\label{p.generic}
For every $f$ that belongs to a $C^1$ residual subset of diffeomorphisms $\cR_0$ and every  periodic point $p$ of $f$, the set $\overline{\cF^u(\Orb(p))}$
is  Lyapunov stable.
\end{proposition}

Recall that the map $\cG^u$ maps a diffeomorphism $g$ to $\G^u(g)$ and is upper semi-continuous by Proposition~\ref{p.Guppercontinuous}. 
Let $\cR_1\subset \cV_k$ be the residual subset of diffeomorphisms which are continuous points of the map $\cG^u$.
For each $1\leq i \leq k$, also consider the map
$\Gamma_i$ from $\cV_k$ to compact subsets of $M$:
$$\Gamma_i(g)=H(p_i(g),g),$$
where $H(p_i(g),g)$ is the homoclinic class of $p_i(g)$.

Because homoclinic classes vary  lower semi-continuously with respect to diffeomorphisms (since they  contain hyperbolic horseshoes), there is a residual subset of diffeomorphisms
$\cR_2\subset \cV_k$ consists of the continuous points of $\Gamma_i$ for every $1\leq i \leq k$.

Now let us take $\cR=\cR_0\cap \cR_1\cap \cR_2\subset \cV_k$. For each $C^1$ diffeomorphism $h\in\cR$, let
$$
S(h)=\{p_1(h),\ldots,p_k(h)\} 
$$
be a skeleton of $h$. 
We are going to show that the  residual set $\cR$ satisfies the condition we need.

\begin{proposition}
Every $C^1$ diffeomorphism $h\in\cR\subset\cV_k$ has exactly $k$ physical measure, each of which is supported on $\overline{\cF^u(\Orb(p_i(h)))}$ for some $i=1,\ldots, k$. Furthermore, the basin of each physical measure covers a full volume subset within a neighborhood of its support.
\end{proposition}

\begin{proof}
For any $C^{1+}$ diffeomorphism $g\in \cV_k$, denote by 
$$\mu_{g,1},\cdots, \mu_{g,k}$$
the ergodic physical measures of $g$. Then $\G^u(g)=\Gibb^u(g)$ is the simplex generated by $\{\mu_{g,1},\cdots, \mu_{g,k}\}. $

Then for any $h\in \cR$, by continuity of $\cG^u$, we see that $\G^u(h)=\cG^u(h)$ is also a simplex of dimension  $m_h\leq k$. In particular, the number of extreme elements of $\G^u(h)$ is at most $k$.

Denote the extreme points of $G^u(h)$ by $\mu_{h,1};\cdots, \mu_{h,m_h}$.
Moreover, take $g_n$ a sequence of $C^{1+}$ diffeomorphisms converge to $h$, by continuity and relabeling if necessary,  we may assume that  $\lim \mu_{g_n,i}=\mu_{h,i}$ for $i=1,2,\ldots, m_h$.

Note that $\mu_{h,i}$ is supported on $\overline{\cF^u(\Orb(p_i(h)))}$. This is because $\mu_{g_n,i}$ is supported on $\overline{\cF^u(\Orb(p_i(g_n)))}=H(p_i(g_n),g_n)$,
and $$\lim_n H(p_i(g_n),g_n)=H(p_i(h),h)\subset \overline{\cF^u(\Orb(p_i(h)))}.$$

Next, we claim that $m_h=k$, in other words, the dimension of $\G^u(h)$ is indeed $k$. Assume that this is not the case. Then we take $m_h<j\leq k$ and take a weak* limit of $\mu_{h,j}=\lim_n\mu_{g_n,j}$. Note that $\mu_{h,j}\in \G^u(h)$ is supported on $\overline{\cF^u(\Orb(p_j(h)))}$ by the discussion above. Take any ergodic component $\tilde{\mu}_{h,j}$ of $\mu_{h,j}$, then $\tilde{\mu}_{h,j}\in \G^u(h)$ by Proposition~\ref{p.Gu} and is still supported on $\overline{\cF^u(\Orb(p_j(h)))}$.  By~\eqref{eq.disjoint}, $\tilde{\mu}_{h,j}\ne \mu_{h,i}$ for every $i=1,\ldots, m_h$. We have thus created a new extreme point of $\G^u(h)$, which is a contradiction.

To finish the proof, we have to shows that each $\mu_{h,i}$ is a physical measure. 

Since $\overline{\cF^u(\Orb(p_i(h)))}$ are Lyapunov stable, we can take $U_i\supset V_i$ disjoint open neighborhoods for each $\overline{\cF^u(\Orb(p_i(h)))}$, such that $f^n(V_i)\subset U_i$ for any $n>0$.

By Theorem~\ref{main.newcriterion}, there is a full volume subset $\Gamma_i\subset V_i$ such that for
any $x\in \Gamma_i$, any limit $\mu$ of the sequence $\frac{1}{n}\sum_{j=0}^{n-1}\delta_{f^j(x)}$
belongs to $\G^u(h)$. %Thus $\mu$ can be written as a combination of $\mu_{h,i}$:
%$$\mu=a_1\mu_{h,1}+\cdots + a_k \mu_{h,k},$$
%where $a_1+\cdots + a_k=1$.
Note that since  $x\in V_i$, we have $f^n(x)\in U_i$ for all $n\geq 1$. As a result, $\mu$ is supported on $U_i$. On the other hand, $\mu_{h,i}$
is the only ergodic measure in $\G^u(h)$ that is supported on $U_i$. It follows that  $\mu=\mu_{h,i}$. 
This implies that Lebesgue almost every point of $x\in V_i$ belongs to the basin of $\mu_{h,i}$.

The proof is complete.
\end{proof}

So far we have shown that $C^1$ generic $h\in\cV_k$ has exactly $k$ physical measures. Take the union of $\cV_k$ over $k=1,\ldots,k_f$, we conclude that $C^1$ generic $h\in\cU$ has at most $k_f$ physical measures. To finish the proof of Theorem~\ref{t.mostlyoontracting}, it remains to show that: 

\begin{lemma}\label{l.fullbasin}
The basins of $\mu_{h,i}$ for $i=1,\cdots, k$ covers a  full volume set.
\end{lemma}
\begin{proof}
Let $\Gamma$ be the full volume subset given by Theorem~\ref{main.newcriterion}. We are going to show
that $\Leb(\Gamma\setminus \bigcup_{i=1}^k Basin(\mu_{h,i}))=0$.

We prove by contradiction. Write $\Lambda=\Gamma\setminus \bigcup_{i=1}^k Basin(\mu_{h,i})$ and suppose that
$$\Leb(\Lambda)>0.$$
Let  $x\in \Lambda$ be a Lebesgue density point of $\Lambda$, which means that for any $r>0$, we have 
$\Leb(B_r(x)\cap \Lambda)>0$. Let $\mu$ be any limit of the sequence $\frac{1}{n}\sum_{i=0}^{n-1}\delta_{f^i(x)}$. Since $\mu\in \G^u(f)$,
$\mu$ can be written as a combination of $\mu_{h,i}$:
$$\mu=a_1\mu_{h,1}+\cdots + a_k \mu_{h,k},$$
where $a_1+\cdots + a_k=1$.

Suppose without loss of generality that $a_1>0$, then $\mu(U_1)\geq a_1$.
Thus there is $n>0$ such that $\frac{1}{n}\sum_{i=0}^{n-1}\delta_{f^i(x)}(V_1)>0$. In particular, there is $0\leq m \leq n-1$ such that
$f^m(x_0)\in V_1$. Take $\vep$ sufficiently small, we have 
$$f^m(B_\vep(x_0))\subset V_1\subset Basin(\mu_{h,1}).$$ 
This shows that  $f^m(B_\vep(x_0)\cap \Lambda)$ intersects with the basin of $\mu_{h,1}$ on a positive volume set. Because the basin of a measure is invariant under iteration
of $f$ and $f^{-1}$, we have $\Leb(\Lambda\cap Basin(\mu_{h,1}))> 0$, which contradicts with the choice of $\Lambda$.
\end{proof}

\section{G states of partially hyperbolic diffeomorphisms}\label{s.1d}
Throughout this section, let $f$ be a $\C^{1+}$ partially hyperbolic diffeomorphism with one dimension center bundle.
We have $\Gibb^u(f)=\G^u(f)$ as before due to Ledrappier~\cite{L84}.

In the proof we will consider another space of invariant probabilities of $f$:
\begin{definition}\label{df.G}
\begin{equation*}
\G^{cu}(f)=\{\mu\in \cM_{\inv}(f): h_\mu(f)\geq \int \log(\det(Tf\mid_{E^{cu}(x)}))d\mu(x)\}
\end{equation*}
where $E^{cu}=E^c\oplus E^u$.
\end{definition}
Note that for some $\mu\in \G^{cu}$, $\mu$ may have negative Lyapunov exponents within $E^{cu}$.

We denote by
$$\G(f)=\G^{u}(f)\cap \G^{cu}(f)$$
and observe that the space $\G(f)$ is non-empty.

\begin{proposition}\label{p.physical}
There is a full volume subset $\Gamma$ such that for any $x\in\Gamma$, any Cesaro limit of the
sequence $\frac{1}{n}\sum_{i=0}^{n-1}\delta_{f^i(x)}$ belongs to $\G(f)$.

\end{proposition}
\begin{proof}
By \cite{CCE}, for $x$ belonging to a full volume subset, any limit of the sequence $\frac{1}{n}\sum_{i=0}^{n-1}\delta_{g^i(x)}$
belongs to $\G^{cu}$. Moreover, by Theorem~\ref{main.newcriterion}, for $x$ belonging to a full volume subset, any limit of the sequence
$\frac{1}{n}\sum_{i=0}^{n-1}\delta_{g^i(x)}$ belongs to $\G^{u}$. We conclude the proof by taking the intersection
of the two full volume subsets.
\end{proof}

\begin{lemma}\label{l.Gexceed}
If $\mu$ is an ergodic Gibbs $u$-state with non-positive center exponent, then $\mu\in \G(f)$.
\end{lemma}
\begin{proof}
We only need to verify that $h_\mu(f)- \int \log(\det(Tf\mid_{E^{cu}(x)}))d\mu(x)\geq 0$.
Since $\mu$ is a Gibbs $u$-state with non-positive center exponent, by Ledrappier and Young\cite{LY85a},
$h_\mu(f)$ coincides with the sum of positive exponent, all of which corresponds to $E^u$. Note that
$\int \log(\det(Tf\mid_{E^{cu}(x)}))d\mu(x)$ equals the sum of all positive exponents  with the non-positive center exponent,
we have
$$h_\mu(f)- \int \log(\det(Tf\mid_{E^{cu}(x)}))d\mu(x)=-\lambda^c_\mu(f)\geq 0.$$
\end{proof}
\begin{remark}\label{rk.Gbad}
This shows that when $\mu$ has negative center exponent,
$$h_\mu(f)- \int \log(\det(Tf\mid_{E^{cu}(x)}))d\mu(x)>0.$$
Let $\nu$ be another Gibbs $u$-state $\nu$ with non-negative center exponent. Then the combination 
$a\mu+(1-a)\nu$ for $a$ sufficiently close to $1$ still belongs to $\G(f)$.
Unlike the space $\G^u(f)$, the extreme elements of $\G(f)$ is not necessarily ergodic.
\end{remark}

Similar to Gibbs $u$-states, we say that a measure $\mu$ is a {\em Gibbs $cu$-state}, if the conditional measures of $\mu$ along the Pesin unstable manifolds (with dimension $\dim(E^{cu})$) are absolutely continuous with respect to the Lebesgue measure there.
\begin{lemma}\label{l.Gcu}
If $\mu\in \G^{cu}(f)$ is an ergodic Gibbs $u$-state with positive center exponent, then $\mu$ is a Gibbs $cu$-state.
In particular, $\mu$ is a physical measure.
\end{lemma}

\begin{proof}
By Ruelle inequality,
$$h_{\mu}(f)\leq \int \log(\det(Tf\mid_{E^{cu}(x)}))d\mu(x).$$

From the definition, $\mu\in \G^{cu}(f)$ means
$$h_{\mu}(f)\geq \int \log(\det(Tf\mid_{E^{cu}(x)}))d\mu(x),$$
Hence we have
$$h_{\mu}(f)= \int \log(\det(Tf\mid_{E^{cu}(x)}))d\mu(x).$$

By Ledrappier \cite{L84}, $\mu$ is a Gibbs $cu$-state.
This implies that the regular points of $\mu$ has full Lebesgue measure along some Pesin unstable manifold. Because the basin of $\mu$ is $s$-saturated, and the stable foliation is absolutely continuous, the
basin of $\mu$ has positive volume. Thus $\mu$ is a physical measure of $f$.
\end{proof}
\begin{remark}\label{rk.ballbasin}
For  ergodic Gibbs $u$-state $\mu\in \G(f)$ with positive center exponent, the  proof above indeed shows that
there is a ball $B$ with $\mu(B)>0$, such that Lebesgue almost every point of $B$ belongs to the basin of $\mu$.
\end{remark}

Let us begin the proof of Theorem~\ref{main.center1d}.
\begin{proof}
Write $\Gibb_e^{u,+}(f)$ $\Gibb_e^{u,0}(f)$ and $\Gibb_e^{u,-}(f)$ the set of ergodic Gibbs $u$-states of $f$ with positive, zero, and negative center
exponent respectively.

Note that every ergodic Gibbs $u$-state with negative center exponent is a physical measure. As a result,  $f$ has at most countably many ergodic Gibbs $u$-states with negative center exponent, denoted by $\mu^-_1,\mu^-_2,\cdots$.

By Lemma~\ref{l.Gcu}, every measure in $\Gibb_e^{u,+}(f)\cap\G^{cu}(f)$ is a physical measure. In particular, there are at most countably many such measures; we denote them by $\mu^+_{p,1},\mu^+_{p,2},\cdots$ and write their collection as $\Gibb^{u,+}_{p}$. From the definition of $\G^{cu}(f)$,
 any  ergodic Gibbs $u$-state $\mu\in\Gibb^{u,+}_e\setminus \G^{cu}(f)$ must satisfy
\begin{equation}\label{eq.negative}
h_\mu(f)-\int \log(\det(Tf\mid_{E^{cu}(x)}))d\mu(x)<0.
\end{equation}

Let $\Gamma$ be the full volume subset given by Theorem~\ref{main.newcriterion} and write
$$\Gamma_1=\Gamma\setminus \cup_{i} Basin(\mu^-_i).$$ If $\Gamma_1$ has zero volume then the proof is finished. Otherwise, we have the following lemma. 

\begin{lemma}\label{l.removenegative}
There is a full volume subset $\tilde{\Gamma}_1\subset \Gamma_1$, such that for any
point $x\in \tilde{\Gamma}_1$ and any  limit $\mu$ of the sequence $\frac{1}{n}\sum_{i=0}^{n-1}\delta_{f^i(x)}$,
the ergodic decomposition of $\mu$ is supported on $\Gibb_e^{u,0}\cup \Gibb_e^{u,+}$.
\end{lemma}

\begin{remark}\label{rk.0cu}
Lemma~\ref{l.removenegative} indeed implies that for any $x\in \tilde{\Gamma}_1$, any
limit $\mu$ of the sequence $\frac{1}{n}\sum_{i=0}^{n-1}\delta_{f^i(x)}$,
the ergodic decomposition of $\mu$ is supported on $\Gibb_e^{u,0}\cup \Gibb_p^{u,+}$.
In fact, by Proposition~\ref{p.physical},
$$h_\mu(f)-\int \log(\det(Tf\mid_{E^{cu}(x)}))d\mu(x)\geq 0.$$
Similar to the proof of Proposition~\ref{p.Gu} on the ergodic component, the entropy function is an affine function, and for $\mu\in \Gibb^{u,0}_e(f)\cup \Gibb^{u,+}_p$,
$$h_\mu(f)-\int \log(\det(Tf\mid_{E^{cu}(x)}))d\mu(x)= 0.$$
Then this observation follows from \eqref{eq.negative}.
\end{remark}

We leave the proof of Lemma~\ref{l.removenegative} to the Appendix, and continue with the proof of Theorem~\ref{main.center1d}.

Let $\Gamma_2=\tilde{\Gamma}_1\setminus \cup_i Basin(\mu^{+}_{p,i})$. Again, if $\Gamma_2$ has zero volume then the proof is finished. Otherwise, we need the following lemma:

\begin{lemma}\label{l.G0}
There is a full volume subset $\tilde{\Gamma}_2\subset \Gamma_2$, such that for any
point $x\in \tilde{\Gamma}_2$ and any limit $\mu$ of the sequence $\frac{1}{n}\sum_{i=0}^{n-1}\delta_{f^i(x)}$,
the ergodic decomposition of $\mu$ is supported on $\Gibb_e^{u,0}$.
\end{lemma}
\begin{proof}
We prove by contradiction. Suppose there is a positive volume subset $\Gamma_3\subset \Gamma_2$ such that for any $x\in \Gamma_3$, there is some limit $\mu$ of the sequence $\frac{1}{n}\sum_{i=0}^{n-1}\delta_{f^i(x)}$ that is not supported on $\Gibb^{u,0}_e$. By Lemma~\ref{l.removenegative} and Remark~\ref{rk.0cu},  $\mu$ can be written as a combination in the following way:
$$\mu=a \mu^{+}_{p,i_x}+(1-a)\nu\, \text{ for some } \mu^{+}_{p,i_x}\in \Gibb_p^{u,+}.$$
where $a>0$ and $\nu$ is a Gibbs $u$-state.

The rest of the proof is similar to Lemma~\ref{l.fullbasin}. Fix $x^\prime$ a Lebesgue density point of $\Gamma_3$ and limit $\mu$ of the sequence $\frac{1}{n}\sum_{i=0}^{n-1}\delta_{f^i(x^\prime)}$ which
is not supported on $\Gibb^{u,0}_e$. And suppose $$\mu=a \mu^{+}_{p,i_{x\prime}}+(1-a)\nu$$
where $a>0$ and $\nu$ is a Gibbs $u$-state. Without loss of generality, we may assume that $i_{x\prime}=1$.

Let $B_1$ be the ball given by Remark~\ref{rk.ballbasin}. Then $\mu(B_{1})\geq a \mu^{+}_{p,1}(B_1)>0$. Thus there is $m>0$ such that $f^m(x^\prime)\subset B_{1}$. Then $f^m(\Gamma_3)\cap Basin(\mu^{+}_{p,1})$ has positive volume.
By the invariance of basin under the iteration of $f^{-1}$, $\Gamma_3$ has non-trivial intersection with the basin
of $\mu^{+}_{p,1}$, which contradicts with the definition of $\Gamma_3\subset \Gamma_2$.
\end{proof}
To Summarize, we have shown that for Lebesgue  almost every $x$: 
\begin{itemize}
\item either $x$ is in the basin of a Gibbs $u$-state with negative center exponent;
\item or $x$ is in the basin of $\mu\in\Gibb_e^{u,+}(f)\cap\G^{cu}(f)$, and $\mu$ is a Gibbs $cu$-state by Lemma~\ref{l.Gcu};
\item or the ergodic decomposition of every Cesaro limit $\mu$ is supported on $\Gibb_e^{u,0}$.
\end{itemize}
The proof of Theorem~\ref{main.center1d} is complete.
\end{proof}

\begin{proof}[Proof of Corollary~\ref{maincor.uniquevanishing}]
Assume that $f$ has no ergodic Gibbs $u$-state with vanishing center exponent, i.e., $\Gibb_e^{u,0}=\emptyset$. In other words, case (iii) of Theorem~\ref{main.center1d} does not happen. In particular, $f$ must have  physical measures, whose basin covers a set with full volume.

Now assume that $f$ has exactly one ergodic Gibbs $u$-state with vanishing center exponent, denote by $\mu$. Prove by contradiction, assume that $f$ has no physical measure. Then we are in the case (iii) of  Theorem~\ref{main.center1d}: for Lebesgue almost every point, every Cesaro limit must be in $\Gibb_e^{u,0} = \{\mu\}$. This means that $\mu$ is a physical measure, a contradiction.  So $f$ must have at least one physical measure.

In this case, the basin of all physical measures of $f$ having full volume follows from the fact that when $\Gibb_e^{u,0} = \{\mu\}$, the set of points points whose Cesaro limit is in $\Gibb_e^{u,0}$ is indeed the basin of $\mu$. 
\end{proof}

\begin{proof}[Proof of Corollary~\ref{maincor.finitephysicalmeasure}.]
Using the notations above, we have $\Gibb_e^{u,0}=\emptyset$. We have to show that both  $\Gibb^{u,-}_{e}$ and $\Gibb^{u,+}_{p}$ are finite sets.\\
\noindent (1). $\Gibb^{u,-}_{e}$ is finite.

Assume that this is not the case. Then take a subsequence $\{n_k\}$ if necessary,  we have $\mu^-_n\xrightarrow{weak*} \mu$ where $\mu$ must be a Gibbs $u$-state thanks to Proposition~\ref{p.Gibbsustates}.   Moreover, the center exponent of $\mu$ is non-positive. Since $\Gibb_e^{u,0}=\emptyset$, $\mu$ must have a ergodic component with negative center exponent, which is a Gibbs $u$-state of $f$. Without loss of generality, we may assume that $\mu^-_1$ is an ergodic component of $\mu$.

The rest of the proof is similar to \cite[Lemma 4.2]{VY1}. We take an unstable disk $I\subset \supp \mu^-_1$ such that $\Leb_I$ almost every $x\in I$ has Pesin stable manifold with dimension $\dim E^{cs}$. Fix $r>0$ small and take  $\Gamma_r\subset I$ with positive $\Leb_I$ measure, such that points in $\Gamma_r$ has Pesin stable manifold with size at least $r$. We have $\bigcup_{x\in\Gamma_r} W^{cs}_r(x)\subset Basin( \mu^-_1)$.

Take typical points $x_{n_k}$ of $\mu^-_{n_k}$ such that $\Leb^u$ almost every points of  $\cF^u_{\loc}(x_{n_k})$ are also typical points of $\mu^-_{n_k}$, and $\cF^u_{\loc}(x_{n_k})$ is close to $I$ for $k$ large enough. Since Pesin stable manifolds are absolutely continuous,  $\bigcup_{x\in\Gamma_r}W^{cs}_r(x)$  intersects with $\cF^u(x_{n_k})$ on a positive $\Leb^u$ subset for $k$ large enough. This shows that $\supp\mu^-_{n_k}$ and $Basin (\mu^-_1)$ has non-trivial intersection, which is a contradiction.

\noindent (2). $\Gibb^{u,+}_{p}$ is finite. 

The proof is similar. We take $\mu^+_{p,n_k}\xrightarrow{k\to\infty}\mu$. $\mu$ must be a Gibbs $u$-state with non-negative center exponent. By~\cite[Lemma 7.5]{Y16}, $\G(f)$ is compact.  Since $\mu^+_{p,n_k}\in \G(f)$, so is $\mu$. The main difference here is, as we have previously observed, that the ergodic components of measures in $\G(f)$ is not necessarily in $\G(f)$. We need to consider the following two cases:\\
\noindent Case (i). $\mu$ has a ergodic component with negative center exponent. Then the same argument used in (1) gives a contradiction.\\
\noindent Case (ii). Every ergodic component of $\mu$ has positive center exponent.

By Ruelle's inequality, for every ergodic component $\mu_x$ of $\mu$, we have 
$$
h_{\mu_x}(f)- \int \log(\det(Tf\mid_{E^{cu}(y)}))d\mu_x(y)\leq 0.
$$
On the other hand, since $\mu\in\G(f)$, we must have
$$
h_{\mu}(f)- \int \log(\det(Tf\mid_{E^{cu}(y)}))d\mu(y)\geq 0.
$$
This implies that 
$$
h_{\mu_x}(f)- \int \log(\det(Tf\mid_{E^{cu}(y)}))d\mu_x(y)= 0
$$
for almost every ergodic component $\mu_x$. In other words, we must have $\mu_x\in \G(f)$. In particular,  $\mu_x$ is a physical measure.

By Remark~\ref{rk.ballbasin}, the basin of $\mu_x$ contains a ball $B$ with size $r>0$. For $k$ large enough, we have a non-trivial intersection of $B$ with the Pesin unstable manifold $\cF^{cu}(x_{n_k})$ at some typical points $x_{n_k}$ of $\mu^+_{p,n_k}$, which is a contradiction.
\end{proof}

\appendix
\section{Proof of Lemma~\ref{l.removenegative}}
Now we prove Lemma~\ref{l.removenegative}, which is similar to the proof of \cite{VY1}[Proposition 6.9] but
much more difficult, since we have to build a covering property using irregular open sets.
\begin{proof}[Proof of Lemma~\ref{l.removenegative}]
Recall that the set $\Gamma_{1}$ is defined as:
$$\Gamma_1=\Gamma\setminus \cup_{i} Basin(\mu^-_i),$$
where $\Gamma$ is the full volume set given by Theorem~\ref{main.newcriterion}. 
 
Suppose by contradiction that there is a positive volume subset $\Lambda \subset \Gamma_1$, such that for every
point $x\in \Lambda$, there is a limit $\mu$ of the sequence $\frac{1}{n}\sum_{i=0}^{n-1}\delta_{f^i(x)}$,
such that the ergodic decomposition of $\mu$ is not supported on $\Gibb_e^{u,0}\cup \Gibb_e^{u,+}$.

Then $\Lambda=\cup \Lambda_i$ where $\Lambda_i$ is the set  of points $x$ such that  there is a limit $\mu_x$ (which must be a Gibbs $u$-state since $\mu_x\in\G(f)$) of the sequence $\frac{1}{n}\sum_{i=0}^{n-1}\delta_{f^i(x)}$, whose ergodic decomposition assigns positive weight on $ \mu^{-}_i$:
$$\mu_x=a_x \mu^{-}_i+(1-a_x)\nu$$
where $\nu$ is a Gibbs $u$-state. We may suppose $\Leb(\Lambda_1)>0$.

By the absolute continuity of unstable foliation, we take an unstable disk $I$ such that
$I\cap \Lambda_1$ has positive $\Leb\mid_I$ measure. We further take $I^*\subset I\cap \Lambda_1$
to be a compact subset and still with positive $\Leb\mid_I$ measure.
Consider $B^I_r(I^*)=\{x\in I: d(x,I^*)< r\}$ the $r$-neighborhood of $I^*$ in $I$. Because $I^*$ is compact, we have
\begin{equation}\label{eq.Dbound}
\cap_{r>0} B^I_r(I^*)=I^*.
\end{equation}
Our goal is to show that $I^*$ has non-trivial intersection with the basin of $\mu_1^-$, which contradicts with the definition of $\Gamma_1$.

For this purpose, take $x$ a typical $\mu^{-}_1$ point such that $x\in \supp(\mu)$, and there is a full volume subset $\Gamma_x$
of $\cF^u_{loc}(x)$ consists of regular $\mu^{-}_1$ points. In particular, every point of $\Gamma_x$ has local Pesin stable
manifold. Take $r>0$ small enough, there is a positive volume subset $\Gamma_r\subset\Gamma_x$ such that each point of
$\Gamma_r$ has Pesin stable manifold $W^{cs}_r$ with radius larger than $r$. By Pesin theory (\cite{P}), the local stable lamination $W^{cs}_{r}(\Gamma_r)$ is absolutely continuous. Passing to a positive volume subset $\Gamma_r$ if necessary, we may assume the stable holonomy map is uniformly absolute continuous: there is $K>0$ such that for any unstable disks $D_1,D_2$ contained in a small neighborhood $U$ of $x$, denote by $\Gamma_r({D_i})$, $i=1,2$ the image of $\Gamma_r$ on $D_i$ under the stable holonomy map, then the  map $\cH^s_{D_1,D_2}: \Gamma_r({D_1}) \to \Gamma_r(D_2)$ induced by the local Pesin stable lamination $W^{cs}_{r}(\Gamma_r)$ has Jacobian bounded in $[1/K,K]$ (we can do this because the Jacobian is uniform bounded for the local unstable lamination of a Pesin block).

Thus by the Pesin theory, we may build a cylinder $\cC$ inside the neighborhood $U$ as:
$$\cC = h(I^{\dim(u)} \times J^{\dim(cs)})$$
where $I^{\dim(u)}$ and $J^{\dim(cs)}$ are the unit disks of $R^{\dim(u)}$ and $R^{\dim(cs)}$ respectively, and $h: I^{\dim(u)} \times J^{\dim(cs)} \to M$ is an embedding  with
$\Gamma_r\subset h(I^{\dim(u)}\times 0)\subset \cF^u_{loc}(x)$
satisfying the following properties: 
\begin{itemize}
\item[(1)] $h({a}\times J^{\dim(cs)})$ is contained in $W^{s}(h(a,0))$ if $h(a,0)\in \Gamma_r$;
\item[(2)] $h(I^{\dim(u)} \times b)$ is a disk $D_b$ contained in $\cF^u_{loc}(h(0,b))$ with radius $R$;
\item[(3)] the volume $\Leb_{D_b}(D_b)$ is uniformly bounded by $K>0$;
\item[(4)] For any two disks $D_{b_1}$ and $D_{b_2}$, there is a holonomy map $h^{cs}$ defined on $\Gamma_r(D_{b_1})$induced by $\{W^{cs}_r(x)\}_{x\in \Gamma_r}$, such that the Jacobian of
the holonomy map between $\Leb\mid\Gamma_r(D_{b_1})$ and $\Leb\mid\Gamma_r(D_{b_2})$ is bounded by $K$ from above and $1/K$ from below.
\end{itemize}
Note that every point in $\Gamma_r(D_{b_1})$ is contained in the basin of $\mu_1^-$. From the property (4) above on the uniform Jacobian of the holonomy map, one has the following lemma:
\begin{lemma}\label{l.uniformbasin}
There is $L>0$ such that for any $b\in J^{\dim(cs)}$, we have 
$$\frac{\Leb^u Basin(\mu^{-}_1)\cap D_b}{\Leb^u D_b}>L.$$
\end{lemma}

 Recall that $I$ is an unstable disk such that $I\cap \Lambda_1$ has positive $\Leb_B$ measure, and $I^*\subset I\cap\Lambda_1$ is compact. The next lemma shows that sets of the form $f^{-n}(D_{b})$ covers $I\cap\Lambda_1$ up to a measure zero set. 

\begin{lemma}\label{l.covering}
For any $n_0>0$, there is a family of disjoint open sets $$\{f^{-n_i}(D_{b_i})\}_{i=1}^\infty\subset I,$$ with $n_i>n_0$, such that $\Leb_I(I^*\setminus \bigcup_i f^{-n_i}(D_{b_i}))=0$.
\end{lemma}
This covering lemma is sufficient for us to finish the proof of Lemma~\ref{l.removenegative}. We may take $r>0$ sufficiently small, such that
$\Leb^u(B_r^I(I^*)\setminus I^*)$ is arbitrarily close to zero.
By uniform expansion on $E^u$, there is $n_r>0$ such that for any $n\geq n_r$, $f^{-n}(D_{b})$ has diameter less than $r$. In particular, the
covering of $I^*$ by  $\{f^{-n_i}(D_{b_i})\}_{i=1}^\infty$ is indeed contained in $B^I_r(I^*)$.

By uniform distortion, the invariance of the basin of $\mu^{-}_1$ and Lemma~\ref{l.uniformbasin} , there is $L_1>0$ does not depend on $r$, such that
$\frac{\Leb^u(Basin(\mu^{-}_1)\cap B^I_r(I^*))}{\Leb^u(\bigcup f^{-n_i}(D_{b_i}))}>L_1$.
In particular, we have
$$\Leb^u(Basin(\mu^{-}_1)\cap B^I_r(I^*))>L_2 \text{ for some } L_2>0 \text{ independent of } r.$$
By the  definition of $I^*\subset \Lambda_1\subset\Lambda\subset\Gamma_1 $ where $\Gamma_1$ is obtained by removing all the basins of $\mu^-_i$, we have
$I^*\cap Basin(\mu^{-}_1)=\emptyset$, thus all the mass in $Basin(\mu^{-}_1)\cap B^I_r(I^*)$
must indeed be contained in $B^I_r(I^*)\setminus I^*$. However, this contradicts with the fact that
$\lim_{r\to 0}\Leb^u(B^I_r(I^*)\setminus I^*)=0$.

So it remains to prove the Lemma~\ref{l.covering}. To simply notation, we may assume that each $D_b\in \cC$ is a ball inside the
unstable leaf, centered at $h(0,b)$ with diameter equal to $1$, i.e., $R=\frac12$.
Denote by $\lambda=\max_{x} \|(Tf^{-1}\mid_{E^u(x)})\|<1$. Then, 
we write the \emph{`bad' neighborhood of a disk $D_b$ at step $n\geq 1$} by
$$D_{b,n}=\{y: y\in B^u_{\frac{1}{2}+\lambda^n}(h(0, b))\setminus B^u_{\frac{1}{2}}(h(0, b))\},$$
which is an annulus with radius $\frac12$ and thickness $\lambda^n$.
The following lemma is immediate:
\begin{lemma}\label{l.away}
For any $x\in \cF^u_{loc}(h(0, b))\setminus (D_{b,n}\cup D_b)$, $d^u(f^n(x),f^n(D_b))>1$.
\end{lemma}

Then there is $L_3>0$ such that for any $D_b$ and any $n>0$, $\Leb^u(D_{b,n})\leq L_3 \lambda^n$.
In particular, by uniform distortion, there is $L_4>0$ such that for any $q>0$,
\begin{equation}\label{eq.badneighborhood}
\sum_{n\geq 1} \frac{\Leb^u(f^{-q}(D_{b,n}))}{\Leb^u(f^{-q}(D_b))}\leq L_4.
\end{equation}

We will select the open sets $\{f^{-n_i}(D_{b_i})\}$ by induction. Take a small $\vep$ neighborhood of $x$, $B_\vep(x)$, such that $B_{\vep}(x)\in \cC$. For every $n\geq n_0$, we denote by $X_n\subset \Lambda_1\cap I^*$ the set of points such that
$f^n(X_n)\subset B_\vep(x)$.

First we take $n_1>n_0$ to be the first time  such that $X_{n_1}\neq \emptyset$ (such $n_1$ exists due to the definition of $\Lambda_1$, and that $I^*$ consists of typical points of $\Lambda_1$), and for every $y\in X_{n_1}$, there is a disk $D_{f^{n_1}(y)} \supset f^{n_1}(y)$.
Take finitely many such disjoint disks with maximal cardinality, and denote them by $D_{b^{n_1}_1},\cdots,D_{b^{n_1}_{k_{n_1}}}$. Note that points in $X_{n_1}$ that is not covered by any $D_{b^{n_1}_i}$ must be contained in the bad neighborhood $D_{b^{n_1}_j,n_1}$ for some $j=1,\ldots, k_1$. Otherwise, by Lemma~\ref{l.away} one can cover $x$ by another ball $f^{-n_1}(D_b)$, such that $D_b$ is disjoint from every $D_{b^{n_1}_j}$, this contradicts with the choice  of $k_1$ having maximal cardinality.  

Suppose we have selected the disks $D_{b^i_1},\cdots, D_{b^i_{k_i}}$ for $n_1\leq i \leq m$.
Denote by $X_{m+1}^\prime =X_{m+1}\setminus \bigcup_{n_1\leq i \leq m; 1\leq j \leq k_i} f^{-i}(D_{b^i_j}\cup D_{b^i_j,m+1-i})$. In other word, we removed from $X_{m+1}$ all the previously selected disks together with their bad neighborhoods.
Note that by Lemma~\ref{l.away}, for any $y\in X_{m+1}^\prime$, and any $n_1\leq i \leq m; 1\leq j \leq k_i$, 
$$d^u(f^{m+1}(y),f^{m+1-i}(D_{b^i_j}))>1.$$
Therefore for every $y\in X_{m+1}^\prime$, there is a disk $D_{f^{m+1}(y)} \supset f^{m+1}(y)$ such that $f^{-(m+1)}(D_{f^{m+1}(y)})$
is disjoint with $\bigcup_{n_1\leq i \leq m; 1\leq j \leq k_i} f^{-i}(D_{b^i_j})$. Take  finitely many such disjoint disks with maximal cardinality, 
and denote them by $D_{b^{m+1}_1},\cdots,D_{b^{m+1}_{k_{m+1}}}$. Similar as before, points in $X'_{m+1}$ that is not covered by any of the selected  balls must be contained in the $(m+1)$ bad neighborhood of one of these balls, otherwise one can fit in another ball, which contradicts with the choice of $k_{m+1}$ having maximal cardinality. 

We need to show that $\Leb^u(I^*\setminus \bigcup_{i\geq n_1; 1\leq j \leq k_i} f^{-i}(D_{b^i_j}))=0$. Suppose
by contradiction that there is $I^*_1=I^*\setminus \bigcup_{i\geq n_1; 1\leq j \leq k_i} f^{-i}(D_{b^i_j})$
with $\Leb^u(I^*_1)>0$. Recall that by definition, $\Leb^u$-almost every point $x\in I^*_1$ enters $B_\vep(x)$ infinitely many
times, thus 
\begin{equation}\label{eq.infinite}
\sum_{m\geq n_1} \Leb^u(X_m\cap I^*_1)=\infty.
\end{equation} 

But from the previous construction, for any $m>n_1$, 
$$X_m\cap I^*_1 \subset \bigcup_{n_1\leq i \leq m-1; 1\leq j \leq k_i} f^{-i}(D_{b^i_j,m-i}).$$
Then $\sum_{m\geq n_1} \Leb^u(X_m\cap I^*_1)\leq \sum_{i\geq n_1, 1\leq j \leq k_i, n\geq 1} \Leb^u(f^{-i}(D_{b^i_j,n}))$.
Also by \eqref{eq.badneighborhood} and the fact that $\{f^{-i}(D_{b^i_j})\}_{i\geq n_1, 1\leq j \leq k_i}$ are all disjoint, we have
$$\sum_{m\geq n_1} \Leb(X_m\cap I^*_1) \leq \sum_{i\geq n_1, 1\leq j \leq k_i}L_4\Leb^u(f^{-i}(D_{b^i_j})\leq L_4\Leb^u(I),$$
which contradicts with \eqref{eq.infinite}.
\end{proof}

\end{document}